%% file: main-COLT.tex
\documentclass[final,12pt]{colt2026} 
\usepackage[T1]{fontenc}
\usepackage{color}
\usepackage{babel}
\usepackage{float}
\usepackage{amsmath}
\usepackage{amssymb}

\makeatletter

\floatstyle{ruled}
\newfloat{algorithm}{tbp}{loa}
\providecommand{\algorithmname}{Algorithm}
\floatname{algorithm}{\protect\algorithmname}

\newtheorem{assumption}{\protect\assumptionname}
\newtheorem{lem}{\protect\lemmaname}
\newtheorem{thm}{\protect\theoremname}
\newtheorem{cor}{\protect\corollaryname}
\newtheorem{rem}{\protect\remarkname}

\@ifundefined{date}{}{\date{}}

\usepackage{enumitem}
\setlist[itemize]{leftmargin=*}
\setlist[enumerate]{leftmargin=*}

\makeatother

\providecommand{\assumptionname}{Assumption}
\providecommand{\corollaryname}{Corollary}

\providecommand{\lemmaname}{Lemma}
\providecommand{\remarkname}{Remark}
\providecommand{\theoremname}{Theorem}

\renewenvironment{proof}[1][\proofname]%
{%
 \par\noindent{\bfseries\upshape #1\ }%
}%
{\jmlrQED}

\title[Random Reshuffling Dominates Stochastic Gradient Descent]{Random Reshuffling Dominates Stochastic Gradient Descent}
\usepackage{times}

\coltauthor{%
 \Name{Zijian Liu} \Email{zl3067@stern.nyu.edu}\\
 \addr Stern School of Business, New York University
}

\begin{document}
\global\long\def\E{\mathbb{E}}%
\global\long\def\N{\mathbb{N}}%
\global\long\def\R{\mathbb{R}}%
\global\long\def\bx{\mathbf{x}}%
\global\long\def\by{\mathbf{y}}%
\global\long\def\bzero{\mathbf{0}}%
\global\long\def\defeq{\triangleq}%
\global\long\def\RR{\textsf{RR}}%
\global\long\def\SS{\textsf{SS}}%
\global\long\def\IG{\textsf{IG}}%
\global\long\def\SGD{\textsf{SGD}}%
\global\long\def\ShufflingSGD{\textsf{Shuffling SGD}}%
\global\long\def\Lavg{\bar{L}}%
\global\long\def\Lmax{\hat{L}}%
\global\long\def\mydots{\dots}%
\global\long\def\Breg{\mathrm{B}}%
\global\long\def\Order{\textsf{Order}}%
\global\long\def\Old{\textsf{Old}}%
\global\long\def\New{\textsf{New}}%

\maketitle
\begin{abstract}
Stochastic Gradient Descent ($\textsf{SGD}$) is one of the most classical optimization algorithms with favorable theoretical guarantees, yet the practical implementation of $\textsf{SGD}$ differs subtly from its well-known form and is often referred to as Shuffling Stochastic Gradient Descent ($\textsf{Shuffling SGD}$). A particularly popular strategy in $\textsf{Shuffling SGD}$ is Random Reshuffling ($\textsf{RR}$), which has achieved great empirical success across numerous experiments. Despite its strong performance, $\textsf{RR}$ has long been considered a heuristic due to a lack of theoretical support. Over the last decade, people have finally established provable convergence rates for $\textsf{RR}$, thus justifying its observed superiority. However, for smooth convex optimization, two clouds over the convergence theory of $\textsf{RR}$ remain to this day. More precisely, according to the current theory, $\textsf{Shuffling SGD}$ under $\textsf{RR}$ converges only when the stepsize is smaller than a threshold proportional to $1/n$, where $n$ is the number of summands in the objective (or the number of data points). Consequently, the optimally tuned theoretical rate of $\textsf{Shuffling SGD}$ under $\textsf{RR}$ is strictly worse than that of $\textsf{SGD}$ when the number of epochs is smaller than another threshold proportional to $n$. These two restrictions heavily limit the applicability of existing theories and leave a critical mismatch with practice. In this work, for the first time, we prove that $\textsf{RR}$ dominates $\textsf{SGD}$ in smooth convex optimization under any reasonable stepsize after any finite number of epochs, thereby addressing a longstanding open question.
\end{abstract}{\footnotetext{In this work, we say that one optimization algorithm dominates another if the order of its convergence rate is no worse than that of the latter and is strictly better in certain regimes.}}


\input{introduction.tex}

\input{preliminary.tex}
\input{algorithm.tex}

\input{conclusion.tex}

\section*{Acknowledgments}

The author thanks the anonymous reviewers for their valuable feedback.

\bibliography{ref}

\clearpage

\appendix
\input{appendix.tex}

\end{document}

%% file: introduction.tex
\section{Introduction\label{sec:Introduction}}

One of the fundamental tasks in machine learning is to optimize functions
in a finite-sum form, i.e., $f(\bx)\defeq\frac{1}{n}\sum_{i=1}^{n}f_{i}(\bx)$.
Among different optimization algorithms, Stochastic Gradient Descent
($\SGD$), proposed in the seminal work of \citet{10.1214/aoms/1177729586},
is arguably one of the most classical methods. Due to its easy implementation
and computational efficiency, $\SGD$ is particularly popular when
$n$ is large, the standard case nowadays. More importantly, the convergence
guarantees of $\SGD$ have been extensively studied, yielding provable
rates in various settings \citep{Solr-KOHA-OAI-TEST:19722,doi:10.1137/16M1080173,lan2020first},
thereby providing a theoretical backbone for $\SGD$.

However, compared with the standard form of $\SGD$ analyzed in theory,
which uniformly samples a function to perform a gradient descent step
at each iteration, the practical implementation differs subtly and
is often referred to as Shuffling Stochastic Gradient Descent ($\textsf{Shuffling SGD}$).
In $\ShufflingSGD$, the optimization procedure is divided into $K$
epochs, and within each epoch, the order in which functions are processed
is determined by a permutation $\pi$ of $\left\{ 1,\mydots,n\right\} $.
A widely implemented strategy for generating $\pi$ is Random Reshuffling
($\RR$), which, in each epoch, independently and uniformly draws
a new permutation from all possible ones.

Although $\ShufflingSGD$ under $\RR$ has achieved great empirical
success across numerous experiments, it has long been considered a
heuristic due to a lack of theoretical support. Over the last decade,
beginning with the pioneering work of \citet{gurbuzbalaban2021random},
people have finally established provable convergence rates for $\textsf{RR}$,
thereby justifying its observed superiority over standard $\SGD$.

In particular, for smooth convex optimization (i.e., each $f_{i}$
is convex and $L$-smooth\footnote{For simplicity, we adopt a uniform smoothness parameter in the discussion,
as in most of the existing literature.}), $\RR$ with a constant stepsize\footnote{We also use a constant stepsize in the discussion for convenience.}
$\eta$ is known to converge in expectation at the rate $\frac{D^{2}}{\eta nK}+\eta^{2}nL\sigma_{\star}^{2}$
(e.g., \citet{NEURIPS2020_c8cc6e90,JMLR:v22:20-1238}), provided that
the stepsize satisfies $\eta\lesssim\frac{1}{nL}$, where $D$ denotes
the distance between the initial point and the optimal solution, and
$\sigma_{\star}^{2}$ is the gradient variance at the minimizer. In
comparison, $\SGD$ under the same setting guarantees the in-expectation
convergence rate $\frac{D^{2}}{\eta nK}+\eta\sigma_{\star}^{2}$ \citep{garrigos2023handbook}
but only requires $\eta\lesssim\frac{1}{L}$. Clearly, $\RR$ converges
faster than $\SGD$ in the regime $\eta\lesssim\frac{1}{nL}$, which
has been recognized as theoretical evidence demonstrating the strong
performance of $\RR$.

Despite the progress discussed above, some important issues remain
unaddressed. The most critical longstanding open question is that
people still do not understand what happens to $\RR$ when $\eta$
falls into the regime $\eta\gtrsim\frac{1}{nL}$. This point is critical
because, given that $n$ is typically large in modern tasks, the threshold
$\frac{1}{nL}$ can be extremely small or even vanish, whereas stepsizes
used in practice are usually at a constant level, thereby leaving
a significant gap between theory and practice. More crucially, even
if one temporarily assumes that the existing rate for $\RR$ mentioned
earlier could be extended to allow $\eta\lesssim\frac{1}{L}$ (though
no such theory has been established), it would still fail to explain
the advantage of $\RR$ over $\SGD$, as the term $\eta^{2}nL\sigma_{\star}^{2}$
for $\RR$ is worse than the term $\eta\sigma_{\star}^{2}$ for $\SGD$
when $\eta\gtrsim\frac{1}{nL}$. This hints that an analysis different
from existing ones may be needed.

Another issue implied by the above discussion is that the optimally
tuned rate for $\RR$ induced by the existing bound is only $\frac{LD^{2}}{K}+(\frac{L\sigma_{\star}^{2}D^{4}}{nK^{2}})^{\frac{1}{3}}$,
which is better than the best tuned rate $\frac{LD^{2}}{nK}+\frac{\sigma_{\star}D}{\sqrt{nK}}$
for $\SGD$ only when $K$ is larger than a threshold proportional
to $n$. Moreover, in the case of $\sigma_{\star}=0$ (i.e., all $f_{i}$'s
share a common optimal solution), the rate of $\RR$ reduces to only
$\frac{LD^{2}}{K}$, which is even worse than the $\frac{LD^{2}}{nK}$
rate of $\SGD$ by a factor of $\frac{1}{n}$.

The above restrictions on the stepsize $\eta$ or the number of epochs
$K$ heavily limit the applicability of existing theories and cannot
fully explain the favorable performance of $\ShufflingSGD$ under
$\RR$ compared with standard $\SGD$. Therefore, we are naturally
led to the following question:
\begin{center}
\textit{In smooth convex optimization, does $\RR$ dominate $\SGD$
without these two restrictions?}
\par\end{center}

\subsection{Our Contributions}

This work provides an affirmative answer to the above question.
\begin{itemize}
\item Concretely, we show that $\ShufflingSGD$ under $\RR$ provably converges
at a rate of $\frac{D^{2}}{\eta nK}+\min\left\{ 1,\eta nL\right\} \eta\sigma_{\star}^{2}$
for any stepsize satisfying $\eta\lesssim\frac{1}{L}$ (see Theorem
\ref{thm:main-RR-cvx} for the formal version with nonuniform smoothness
parameters and dynamic stepsizes that depend on the epoch number).
We highlight that this rate is not only the first provable result
for $\RR$ that allows $\eta\gtrsim\frac{1}{nL}$, but also provably
dominates the $\frac{D^{2}}{\eta nK}+\eta\sigma_{\star}^{2}$ bound
of $\SGD$ under any stepsize $\eta\lesssim\frac{1}{L}$. It is noteworthy
that this rate is not merely a simple extension of the previously
best bound for $\RR$, since the latter is slower than $\SGD$ when
$\eta\gtrsim\frac{1}{nL}$, as discussed before.
\item Consequently, the optimally tuned rate for $\ShufflingSGD$ under
$\RR$ is improved to $\frac{LD^{2}}{nK}+\min\{\frac{\sigma_{\star}D}{\sqrt{nK}},(\frac{L\sigma_{\star}^{2}D^{4}}{nK^{2}})^{\frac{1}{3}}\}$
(see Corollary \ref{cor:main-RR-cvx-constant-stepsize} for the formal
version with nonuniform smoothness parameters), which dominates the
best tuned bound $\frac{LD^{2}}{nK}+\frac{\sigma_{\star}D}{\sqrt{nK}}$
of $\SGD$ for any finite $K$. Moreover, the rate reduces to $\frac{LD^{2}}{nK}$
when $\sigma_{\star}=0$, improving upon the best known result $\frac{LD^{2}}{K}$
by a factor of $\frac{1}{n}$.
\end{itemize}
In summary, for the first time, we prove that $\textsf{RR}$ dominates
standard $\textsf{SGD}$ in smooth convex optimization under any reasonable
stepsize after any finite number of epochs, resolving a longstanding
open question.

\subsection{Related Work}

We provide a brief overview of $\ShufflingSGD$ under $\RR$ and defer
further details to Appendix \ref{sec:Additional-Related-Work}.

Over the past decades, the effectiveness of $\RR$ has been reported
in many works (e.g., \citet{bottou2009curiously,Bottou2012,Bengio2012}).
However, the theoretical understanding of it has long lagged behind.
The first breakthrough is by \citet{gurbuzbalaban2021random}, which
provides the first theoretical evidence that $\RR$ can beat $\SGD$
in smooth strongly convex optimization under certain additional assumptions.
Since then, $\RR$ has been extensively studied (e.g., \citet{pmlr-v97-nagaraj19a,pmlr-v97-haochen19a}).
To date, the best known rate in smooth convex optimization is $\frac{D^{2}}{\eta nK}+\eta^{2}nL\sigma_{\star}^{2}$
under the condition $\eta\lesssim\frac{1}{nL}$ \citep{NEURIPS2020_c8cc6e90,JMLR:v22:20-1238}.
The only existing lower bound in smooth convex optimization is $(\frac{L\sigma_{\star}^{2}D^{4}}{nK^{2}})^{\frac{1}{3}}$,
due to \citet{pmlr-v202-cha23a}, which holds for constant stepsizes
$\eta\lesssim\frac{1}{nL}$ and a large number of epochs at least
satisfying $K\gtrsim\frac{nL^{2}D^{2}}{\sigma_{\star}^{2}}$.

%% file: preliminary.tex
\section{Preliminary}

\paragraph{Notation.}

$\N$ denotes the set of natural numbers (excluding $0$). Given $n\in\N$,
we write $\left[n\right]\defeq\left\{ 1,\mydots,n\right\} $. $\left\langle \cdot,\cdot\right\rangle $
is the standard Euclidean inner product, and $\left\Vert \cdot\right\Vert \defeq\sqrt{\left\langle \cdot,\cdot\right\rangle }$
is the $\ell_{2}$ norm. Given a real-valued differentiable function
$h:\R^{d}\to\R$, $\nabla h(\bx)$ denotes the gradient at $\bx\in\R^{d}$.
The Bregman divergence induced by $h$ is defined as $\Breg_{h}(\bx,\by)\defeq h(\bx)-h(\by)-\left\langle \nabla h(\by),\bx-\by\right\rangle $,
which is nonnegative if $h$ is additionally convex.

\paragraph{Objective.}

We study the following finite-sum optimization problem in this work
\[
\inf_{\bx\in\R^{d}}f(\bx)\defeq\frac{1}{n}\sum_{i=1}^{n}f_{i}(\bx),
\]
where $n\in\N$ and each $f_{i}:\R^{d}\to\R$ is differentiable.
\begin{rem}
To ease notation, hereinafter we use $\Breg\defeq\Breg_{f}$ and $\Breg_{i}\defeq\Breg_{f_{i}}$
to denote the Bregman divergences induced by $f$ and $f_{i}$, respectively. 
\end{rem}

\paragraph{Assumptions.}

Our analysis relies on the following three assumptions.
\begin{assumption}[Minimizer]
\label{assu:minimizer}$\exists\bx_{\star}\in\R^{d}$ such that $f_{\star}\defeq f(\bx_{\star})=\inf_{\bx\in\R^{d}}f(\bx)\in\R$.
\end{assumption}
\begin{assumption}[Convexity]
\label{assu:convexity}Each $f_{i}$ is convex.
\end{assumption}
\begin{assumption}[Smoothness]
\label{assu:smoothness}Each $f_{i}$ is $L_{i}$-smooth, i.e., $\exists L_{i}>0$
such that $\left\Vert \nabla f_{i}(\bx)-\nabla f_{i}(\by)\right\Vert \leq L_{i}\left\Vert \bx-\by\right\Vert ,\forall\bx,\by\in\R^{d}$.
\end{assumption}
All three of the above assumptions are standard and commonly adopted
in the literature \citep{Solr-KOHA-OAI-TEST:19722,nesterov2018lectures,doi:10.1137/16M1080173,lan2020first}.
Notably, we do not impose any assumptions on the difference between
the individual gradient $\nabla f_{i}$ and the full gradient $\nabla f$,
such as the popular finite variance condition.

Next, we introduce three more notations to simplify the expressions
in the subsequent sections. $\sigma_{\star}^{2}$ denotes the variance
of the gradient at the minimizer $\bx_{\star}$, and $\Lavg$ (resp.
$\Lmax$) represents the average (resp. maximum) smoothness parameter,
i.e.,
\begin{eqnarray*}
\sigma_{\star}^{2}\defeq\frac{1}{n}\sum_{i=1}^{n}\left\Vert \nabla f_{i}(\bx_{\star})\right\Vert ^{2}, & \Lavg\defeq\frac{1}{n}\sum_{i=1}^{n}L_{i}, & \Lmax\defeq\max_{i\in\left[n\right]}L_{i}.
\end{eqnarray*}
We note that the quantity $\sigma_{\star}^{2}$ is widely used in
prior works on shuffling gradient methods (e.g., \citet{8514028,NEURIPS2020_c8cc6e90,JMLR:v22:20-1238})
and remains invariant even when $f$ has multiple minimizers (see
Lemma 4.17 of \citet{garrigos2023handbook}). In particular, $\sigma_{\star}^{2}=0$
corresponds to the case in which all $f_{i}$ share a common optimal
solution.

To finish this section, we state a classical result in convex optimization,
known as the co-coercivity property of smooth convex functions, which
serves as a key tool in our analysis. As for its proof, see, for example,
Theorem 2.15 of \citet{nesterov2018lectures}.
\begin{lem}[Co-coercivity]
\label{lem:co-coercivity-cvx}Let $h:\R^{d}\to\R$ be a differentiable
convex function that is also $L$-smooth, then we have, for any $\bx,\by\in\R^{d}$,
\begin{eqnarray*}
\left\Vert \nabla h(\bx)-\nabla h(\by)\right\Vert ^{2}\leq2L\Breg_{h}(\bx,\by) & \text{and} & \left\Vert \nabla h(\bx)-\nabla h(\by)\right\Vert ^{2}\leq L\left\langle \nabla h(\bx)-\nabla h(\by),\bx-\by\right\rangle .
\end{eqnarray*}
\end{lem}

%% file: algorithm.tex
\section{Shuffling Stochastic Gradient Descent}

\begin{algorithm}[H]
\caption{\label{alg:Shuffling-SGD}Shuffling Stochastic Gradient Descent ($\protect\ShufflingSGD$)}

\textbf{Input:} initial point $\bx_{1}^{1}\in\R^{d}$, stepsize $\eta_{k}>0$

\textbf{for} $k=1$ \textbf{to} $K$ \textbf{do}

$\quad$Generate a permutation $\pi_{k}$ of $\left[n\right]$

$\quad$\textbf{for} $i=1$ \textbf{to} $n$ \textbf{do}

$\qquad$$\bx_{k}^{i+1}=\bx_{k}^{i}-\eta_{k}\nabla f_{\pi_{k}^{i}}(\bx_{k}^{i})$

$\quad$\textbf{end for}

$\quad$$\bx_{k+1}^{1}=\bx_{k}^{n+1}$

\textbf{end for}
\end{algorithm}

The method studied in this work, Shuffling Stochastic Gradient Descent
($\ShufflingSGD$), is given in Algorithm \ref{alg:Shuffling-SGD}.
Compared with the standard $\SGD$ algorithm, which uniformly samples
a function to process at each step, $\ShufflingSGD$ determines the
order in which functions are passed in each epoch based on a permutation.
In particular, three strategies for generating permutations are popular
in practice, as illustrated in the following examples, where $S_{n}$
denotes the symmetric group of $\left[n\right]$.
\begin{example}[Random Reshuffling ($\RR$)]
\label{exa:RR}Each $\pi_{k}$ is drawn independently and uniformly
from $S_{n}$.
\end{example}
\begin{example}[Single Shuffling ($\SS$)]
\label{exa:SS}Each $\pi_{k}=\pi$, a permutation drawn uniformly
from $S_{n}$.
\end{example}
\begin{example}[Incremental Gradient ($\IG$)]
\label{exa:IG}Each $\pi_{k}=\pi$, a deterministic permutation from
$S_{n}$. 
\end{example}

\subsection{New Rate for $\protect\RR$}

We are now ready to provide the main result, Theorem \ref{thm:main-RR-cvx},
a new convergence rate for $\RR$.
\begin{thm}
\label{thm:main-RR-cvx}Under Assumptions \ref{assu:minimizer}, \ref{assu:convexity},
and \ref{assu:smoothness}, suppose $\RR$ is employed with $\eta_{k}\leq\frac{1}{6\Lmax},\forall k\in\left[K\right]$,
let
\begin{eqnarray}
\bar{\bx}_{K}\defeq\sum_{k=1}^{K}\sum_{i=1}^{n}\frac{\eta_{k}}{nH_{K}}\bx_{k}^{i} & \text{where} & H_{K}\defeq\sum_{k=1}^{K}\eta_{k},\label{eq:avg-x}
\end{eqnarray}
then $\ShufflingSGD$ (Algorithm \ref{alg:Shuffling-SGD}) guarantees
that
\[
\E\left[f(\bar{\bx}_{K})-f_{\star}\right]\leq\frac{6\left\Vert \bx_{1}^{1}-\bx_{\star}\right\Vert ^{2}}{n\sum_{k=1}^{K}\eta_{k}}+51\min\left\{ \frac{\sum_{k=1}^{K}\eta_{k}^{2}}{\sum_{k=1}^{K}\eta_{k}},\frac{\sum_{k=1}^{K}\eta_{k}^{3}n\Lavg}{\sum_{k=1}^{K}\eta_{k}}\right\} \sigma_{\star}^{2}.
\]
\end{thm}
\begin{rem}
We make no effort to optimize the constants in the bounds obtained
in this work.
\end{rem}
\begin{proof}
The proof is deferred to Subsection \ref{subsec:Proof}.
\end{proof}

To the best of our knowledge, in smooth convex optimization, Theorem
\ref{thm:main-RR-cvx} shows the first convergence rate for $\RR$
under any reasonable stepsize, i.e., $\eta_{k}\lesssim1/\Lmax$, thereby
improving existing results in different aspects. Previously, the best
known convergence rate for $\RR$ is $\frac{D^{2}}{n\sum_{k=1}^{K}\eta_{k}}+\frac{\sum_{k=1}^{K}\eta_{k}^{3}n\Lavg\sigma_{\star}^{2}}{\sum_{k=1}^{K}\eta_{k}}$
(where $D\defeq\left\Vert \bx_{1}^{1}-\bx_{\star}\right\Vert $) under
the condition $\eta_{k}\lesssim1/(n\sqrt{\Lavg\Lmax})$ \citep{pmlr-v235-liu24cg,NEURIPS2024_84d39572}.
However, this requirement on $\eta_{k}$ is highly unrealistic, as
$n$ is typically pretty large in modern machine learning tasks, meaning
that the stepsize has to be extremely small or even vanishing. This
contradicts the constant-level stepsizes commonly used in practice
and leads to a gap between theory and practice.

In comparison, Theorem \ref{thm:main-RR-cvx} not only allows the
stepsize to lie in the constant regime $\eta_{k}\lesssim1/\Lmax$
but also shows a fundamental improvement rather than a mere extension
of the previously best known rate, since the latter becomes slower
than $\SGD$ once $\eta_{k}\gtrsim1/(n\Lavg)$ (the rate of $\SGD$
is stated below), while our Theorem \ref{thm:main-RR-cvx} never does. 

In addition, compared with the $\frac{D^{2}}{n\sum_{k=1}^{K}\eta_{k}}+\frac{\sum_{k=1}^{K}\eta_{k}^{2}\sigma_{\star}^{2}}{\sum_{k=1}^{K}\eta_{k}}$
rate of standard $\SGD$ \citep{garrigos2023handbook}, our Theorem
\ref{thm:main-RR-cvx} is never worse for any reasonable stepsize
(i.e., $\eta_{k}\lesssim1/\Lmax$), and is strictly better when $\eta_{k}\lesssim1/(n\Lavg)$.
This feature has an important implication: if both methods employ
their own optimally tuned stepsizes, $\ShufflingSGD$ under $\RR$
provably achieves a better upper bound that dominates the $\frac{\Lmax D^{2}}{nK}+\frac{\sigma_{\star}D}{\sqrt{nK}}$
rate of $\SGD$ after any finite number of epochs, as evidenced by
Corollary \ref{cor:main-RR-cvx-constant-stepsize} below.
\begin{cor}
\label{cor:main-RR-cvx-constant-stepsize}Under the same setting
as in Theorem \ref{thm:main-RR-cvx}, with the optimally tuned constant
stepsize $\eta_{k}=\eta_{\star},\forall k\in\left[K\right]$, where
$\eta_{\star}\leq\frac{1}{6\Lmax}$, $\ShufflingSGD$ (Algorithm \ref{alg:Shuffling-SGD})
guarantees that
\[
\E\left[f(\bar{\bx}_{K})-f_{\star}\right]\lesssim\frac{\Lmax\left\Vert \bx_{1}^{1}-\bx_{\star}\right\Vert ^{2}}{nK}+\min\left\{ \frac{\sigma_{\star}\left\Vert \bx_{1}^{1}-\bx_{\star}\right\Vert }{\sqrt{nK}},\left(\frac{\Lavg\sigma_{\star}^{2}\left\Vert \bx_{1}^{1}-\bx_{\star}\right\Vert ^{4}}{nK^{2}}\right)^{\frac{1}{3}}\right\} .
\]
\end{cor}
\begin{proof}
The proof is deferred to Subsection \ref{subsec:Proof}.
\end{proof}

In contrast, the known optimally tuned rate for $\RR$ is only $\frac{\sqrt{\Lavg\Lmax}D^{2}}{K}+(\frac{\Lavg\sigma_{\star}^{2}D^{4}}{nK^{2}})^{\frac{1}{3}}$.
To better understand the differences among these optimally tuned rates,
let $\Order_{\SGD}(K)$ and $\Order_{\RR}^{\Old}(K)$ denote the dominant
terms in the existing optimally tuned rates for $\SGD$ and $\ShufflingSGD$
under $\RR$, respectively, i.e.,
\[
\Order_{\SGD}(K)\defeq\max\left\{ \frac{\Lmax D^{2}}{nK},\frac{\sigma_{\star}D}{\sqrt{nK}}\right\} \enspace\text{and}\enspace\Order_{\RR}^{\Old}(K)\defeq\max\left\{ \frac{\sqrt{\Lavg\Lmax}D^{2}}{K},\left(\frac{\Lavg\sigma_{\star}^{2}D^{4}}{nK^{2}}\right)^{\frac{1}{3}}\right\} .
\]
Similarly, $\Order_{\RR}^{\New}(K)$ denotes the dominant term in
the rate obtained in Corollary \ref{cor:main-RR-cvx-constant-stepsize},
i.e.,
\[
\Order_{\RR}^{\New}(K)\defeq\max\left\{ \frac{\Lmax D^{2}}{nK},\min\left\{ \frac{\sigma_{\star}D}{\sqrt{nK}},\left(\frac{\Lavg\sigma_{\star}^{2}D^{4}}{nK^{2}}\right)^{\frac{1}{3}}\right\} \right\} .
\]

In the nondegenerate case $\sigma_{\star}\neq0$ (which implies that
$n\geq2$), the following comparisons hold:
\[
\begin{cases}
\Order_{\RR}^{\New}(K)=\Order_{\SGD}(K)\overset{(a)}{\leq}\Order_{\RR}^{\Old}(K), & K\leq\frac{n\Lavg^{2}D^{2}}{\sigma_{\star}^{2}},\\
\Order_{\RR}^{\New}(K)<\Order_{\SGD}(K)\overset{(b)}{\leq}\Order_{\RR}^{\Old}(K), & \frac{n\Lavg^{2}D^{2}}{\sigma_{\star}^{2}}<K\leq\frac{n\Lavg\Lmax D^{2}}{\sigma_{\star}^{2}},\\
\Order_{\RR}^{\New}(K)\overset{(c)}{\leq}\Order_{\RR}^{\Old}(K)<\Order_{\SGD}(K), & \frac{n\Lavg\Lmax D^{2}}{\sigma_{\star}^{2}}<K\leq\frac{n\Lavg^{1/2}\Lmax^{3/2}D^{2}}{\sigma_{\star}^{2}},\\
\Order_{\RR}^{\New}(K)=\Order_{\RR}^{\Old}(K)<\Order_{\SGD}(K), & K>\frac{n\Lavg^{1/2}\Lmax^{3/2}D^{2}}{\sigma_{\star}^{2}},
\end{cases}
\]
where $(a)$ becomes an equality if and only if $K=\frac{n\Lavg^{2}D^{2}}{\sigma_{\star}^{2}}$
and $\Lavg=\Lmax$, $(b)$ becomes an equality if and only if $K=\frac{n\Lavg\Lmax D^{2}}{\sigma_{\star}^{2}}$,
and $(c)$ becomes an equality if and only if $K=\frac{n\Lavg^{1/2}\Lmax^{3/2}D^{2}}{\sigma_{\star}^{2}}$.
As one can see, unlike the existing optimally tuned bound for $\RR$
in the literature, which can be slower than standard $\SGD$ when
$K\leq\frac{n\Lavg\Lmax D^{2}}{\sigma_{\star}^{2}}$, the new result
in Corollary \ref{cor:main-RR-cvx-constant-stepsize} is never worse
than the optimally tuned rate of standard $\SGD$ and strictly beats
it once $K>\frac{n\Lavg^{2}D^{2}}{\sigma_{\star}^{2}}$, thereby improving
the threshold from $\frac{n\Lavg\Lmax D^{2}}{\sigma_{\star}^{2}}$
to $\frac{n\Lavg^{2}D^{2}}{\sigma_{\star}^{2}}$.

Moreover, in the special case where all $f_{i}$'s share a common
minimizer, or equivalently when $\sigma_{\star}=0$, Corollary \ref{cor:main-RR-cvx-constant-stepsize}
achieves the same $\frac{\Lmax D^{2}}{nK}$ rate as $\SGD$. However,
the prior bound for $\RR$ can only be reduced to a slower rate in
the order of $\frac{\sqrt{\Lavg\Lmax}D^{2}}{K}$ due to the stepsize
restriction $\eta_{k}\lesssim1/(n\sqrt{\Lavg\Lmax})$, as discussed
before. Formally, once $n\geq2$, we have
\[
\Order_{\RR}^{\New}(K)=\Order_{\SGD}(K)<\Order_{\RR}^{\Old}(K),\forall K\in\N.
\]

In summary, Theorem \ref{thm:main-RR-cvx} and Corollary \ref{cor:main-RR-cvx-constant-stepsize}
together imply that $\RR$ dominates $\SGD$ in smooth convex optimization
under any reasonable stepsize after any finite number of epochs.

\section{Theoretical Analysis}

In this section, we lay the groundwork for proving Theorem \ref{thm:main-RR-cvx}
and complete its proof at the end. The section is organized into four
parts. First, we provide the two most important lemmas in Subsection
\ref{subsec:Core-Lemmas}. Next, in Subsection \ref{subsec:Bound-I},
we establish an upper bound in Theorem \ref{thm:I-rate}, which indicates
that $\ShufflingSGD$ under $\RR$ never converges more slowly than
$\SGD$ under any reasonable stepsize. Then, Theorem \ref{thm:II-rate}
in Subsection \ref{subsec:Bound-II} presents an alternative convergence
rate for $\ShufflingSGD$ under $\RR$, which demonstrates that Algorithm
\ref{alg:Shuffling-SGD} under $\RR$ provably converges faster than
$\SGD$ when the stepsize is sufficiently small. Finally, in Subsection
\ref{subsec:Proof}, we conclude Theorem \ref{thm:main-RR-cvx} and
then use it to prove Corollary \ref{cor:main-RR-cvx-constant-stepsize}.

\subsection{Two Core Lemmas\label{subsec:Core-Lemmas}}

This subsection contains two core lemmas, both of which are critical
to our analysis. 

Before presenting the lemmas, we introduce two notions. Given a permutation
$\pi$ of $\left[n\right]$ and two indices $i,j\in\left[n\right]$
satisfying $j\leq i$, we define the following new permutation
\begin{equation}
\pi(i,j)\defeq\left(\pi^{1},\mydots,\pi^{j-1},\pi^{i},\pi^{j+1},\mydots,\pi^{i-1},\pi^{j},\pi^{i+1},\mydots,\pi^{n}\right).\label{eq:main-virtual-pi}
\end{equation}
In words, $\pi(i,j)$ is the permutation generated by exchanging the
elements $\pi_{i}$ and $\pi_{j}$ in $\pi$. Equipped with the notion
of $\pi(i,j)$, we introduce the following virtual sequence, for any
given $k\in\left[K\right]$,
\begin{eqnarray}
 & \bx_{k}^{l+1}(i,j)\defeq\bx_{k}^{l}(i,j)-\eta_{k}\nabla f_{\pi_{k}^{l}(i,j)}(\bx_{k}^{l}(i,j)),\forall l\in\left[n\right], & \text{where}\enspace\bx_{k}^{1}(i,j)\defeq\bx_{k}^{1}.\label{eq:main-virtual-x}
\end{eqnarray}
This means that the sequence $\bx_{k}^{l}(i,j),\forall l\in\left[n+1\right]$
denotes the trajectory of the $k$-th epoch starting from $\bx_{k}^{1}$,
produced by Algorithm \ref{alg:Shuffling-SGD}, but under the permutation
$\pi_{k}(i,j)$. This virtual iterate can be viewed as a coupled sequence
of the real output and, to the best of our knowledge, was first introduced
by \citet{NEURIPS2021_107030ca}. It plays a fundamental role in our
proof, as will become clear.

With these two concepts in hand, we proceed to state the two core
lemmas. The first is Lemma \ref{lem:core-function}, which is based
on Lemma 2 of \citet{NEURIPS2021_107030ca}. For completeness, we
reproduce the proof of Lemma \ref{lem:core-function} in Appendix
\ref{sec:Core-Lemmas}.
\begin{lem}
\label{lem:core-function}Given an arbitrary finite-sum function $\ell(\bx)=\frac{1}{n}\sum_{i=1}^{n}\ell_{i}(\bx)$,
suppose $\RR$ is employed, then for any $k\in\left[K\right]$ and
$i\in\left[n\right]$, $\ShufflingSGD$ (Algorithm \ref{alg:Shuffling-SGD})
guarantees that
\[
\E\left[\ell(\bx_{k}^{i})-\ell_{\pi_{k}^{i}}(\bx_{k}^{i})\right]=\frac{1}{n}\sum_{j<i}\E\left[\ell_{\pi_{k}^{i}}(\bx_{k}^{i}(i,j))-\ell_{\pi_{k}^{i}}(\bx_{k}^{i})\right],
\]
where $\bx_{k}(i,j)$ is defined in (\ref{eq:main-virtual-x}).
\end{lem}
In the analysis of shuffling gradient methods, a well-known major
challenge is to properly bound the term $\E\left[f(\bx_{k}^{i})-f_{\pi_{k}^{i}}(\bx_{k}^{i})\right]$,
unlike in $\SGD$, which no longer equals $0$ due to the nature of
shuffling-based algorithms. Lemma \ref{lem:core-function} provides
a possible approach by relating the term we want to control (in a
slightly more general form, applicable to any finite-sum function
$\ell$) to another quantity involving the virtual sequence introduced
earlier in (\ref{eq:main-virtual-x}).

For the convenience of the discussion, temporarily assume $\ell_{i}=f_{i}$
in Lemma \ref{lem:core-function}. Then, under the smoothness assumption,
one would expect the difference between $\ell_{\pi_{k}^{i}}(\bx_{k}^{i}(i,j))$
and $\ell_{\pi_{k}^{i}}(\bx_{k}^{i})$ to be small whenever $\bx_{k}^{i}(i,j)$
and $\bx_{k}^{i}$ are close. This observation naturally leads us
to the other core Lemma \ref{lem:core-stability} stated below.
\begin{lem}
\label{lem:core-stability}Under Assumptions \ref{assu:convexity}
and \ref{assu:smoothness}, suppose $\eta_{k}\leq\frac{2}{\Lmax},\forall k\in\left[K\right]$,
then for any $k\in\left[K\right]$, $i\in\left[n\right]$, and $j\in\left[i-1\right]$,
$\ShufflingSGD$ (Algorithm \ref{alg:Shuffling-SGD}) guarantees that
\[
\left\Vert \bx_{k}^{i}(i,j)-\bx_{k}^{i}\right\Vert \leq\eta_{k}\left\Vert \nabla f_{\pi_{k}^{i}}(\bx_{k}^{j})-\nabla f_{\pi_{k}^{j}}(\bx_{k}^{j})\right\Vert ,
\]
where $\bx_{k}(i,j)$ is defined in (\ref{eq:main-virtual-x}).
\end{lem}
Lemma \ref{lem:core-stability} quantifies how close the virtual iterate
and the true trajectory can be, under the widely required condition
of $\eta_{k}\leq2/\Lmax$ in smooth optimization. Although the inequality
does not directly offer a bound on the distance between $\bx_{k}^{i}(i,j)$
and $\bx_{k}^{i}$ that depends only on deterministic terms (e.g.,
the stepsize $\eta_{k}$), it is sufficient for our proof when combined
with a careful analysis.

The proof of Lemma \ref{lem:core-stability} is given in Appendix
\ref{sec:Core-Lemmas} and relies on the co-coercivity property (i.e.,
Lemma \ref{lem:co-coercivity-cvx}), which is closely related to the
nonexpansiveness of the update rule in gradient methods \citep{Solr-KOHA-OAI-TEST:19722,nesterov2018lectures}.

\subsection{Bound I: Never Worse than SGD under Reasonable Stepsize\label{subsec:Bound-I}}

In this subsection, we give the first bound for $\ShufflingSGD$ under
$\RR$, stated in Theorem \ref{thm:I-rate} below.
\begin{thm}
\label{thm:I-rate}Under Assumptions \ref{assu:minimizer}, \ref{assu:convexity},
and \ref{assu:smoothness}, suppose $\RR$ is employed with $\eta_{k}\leq\frac{1}{6\Lmax},\forall k\in\left[K\right]$,
then $\ShufflingSGD$ (Algorithm \ref{alg:Shuffling-SGD}) guarantees
that
\[
\E\left[f(\bar{\bx}_{K})-f_{\star}\right]\leq\frac{\left\Vert \bx_{1}^{1}-\bx_{\star}\right\Vert ^{2}}{2n\sum_{k=1}^{K}\eta_{k}}+\frac{51\sum_{k=1}^{K}\eta_{k}^{2}\sigma_{\star}^{2}}{\sum_{k=1}^{K}\eta_{k}},
\]
where $\bar{\bx}_{K}$ is defined in (\ref{eq:avg-x}).
\end{thm}

\paragraph{Discussion on Theorem \ref{thm:I-rate}.}

To the best of our knowledge, Theorem \ref{thm:I-rate} offers the
first theoretical evidence that $\ShufflingSGD$ under $\RR$ shares
surprising similarities with $\SGD$, as reflected in the two aspects
elaborated below.

First, Theorem \ref{thm:I-rate} states that, similar to $\SGD$,
$\ShufflingSGD$ under $\RR$ does converge under any reasonable stepsize
(i.e., $\eta_{k}\lesssim1/\Lmax$). In contrast, as far as we know,
all prior works that provide provable rates of multi-epoch $\RR$
for smooth convex optimization require the stepsize $\eta_{k}$ to
be smaller than a threshold proportional to $1/n$, with only one
exception \citep{pmlr-v97-nagaraj19a}, which, however, assumes each
$f_{i}$ to be additionally Lipschitz, thereby limiting the applicability
of their theory and even excluding common quadratic optimization problems
over $\R^{d}$.

Second, we highlight that Theorem \ref{thm:I-rate} gives the same
convergence upper bound (up to constant factors) as $\SGD$ \citep{garrigos2023handbook}
in smooth convex optimization, while allowing a stepsize that depends
on the epoch number. This result thus fills a gap in the literature.

Putting these together, Theorem \ref{thm:I-rate} indicates that,
for smooth convex optimization, $\RR$ under any reasonable stepsize
never converges more slowly than $\SGD$.

\paragraph{Analysis.}

In the following, we present the analysis for Theorem \ref{thm:I-rate}
and finally prove it. The core idea underlying the proof is, as one
might expect, to analyze $\ShufflingSGD$ in a manner analogous to
$\SGD$. In other words, we aim to quantify the progress made by Algorithm
\ref{alg:Shuffling-SGD} at each iteration. Although this perspective
is natural, it has been less explored in prior studies. The main reason
is that, as discussed earlier, the permutation in $\ShufflingSGD$
causes the most important property of $\SGD$, unbiasedness, to no
longer hold. To overcome this barrier, we develop a novel analysis
that avoids any additional assumptions, such as Lipschitz continuity
considered in \citet{pmlr-v97-nagaraj19a}.

We start with the following Lemma \ref{lem:I-descent}, a standard
step in characterizing the per-iterate progress of $\ShufflingSGD$
(or $\SGD$). The proof of Lemma \ref{lem:I-descent} follows directly
from expanding both sides. To make the work self-contained, we include
it in Appendix \ref{sec:Bound-I}.
\begin{lem}
\label{lem:I-descent}Under Assumption \ref{assu:minimizer}, for
any $k\in\left[K\right]$ and $i\in\left[n\right]$, $\ShufflingSGD$
(Algorithm \ref{alg:Shuffling-SGD}) guarantees that
\[
f_{\pi_{k}^{i}}(\bx_{k}^{i})-f_{\pi_{k}^{i}}(\bx_{\star})=\frac{\left\Vert \bx_{k}^{i}-\bx_{\star}\right\Vert ^{2}-\left\Vert \bx_{k}^{i+1}-\bx_{\star}\right\Vert ^{2}}{2\eta_{k}}+\frac{\eta_{k}}{2}\left\Vert \nabla f_{\pi_{k}^{i}}(\bx_{k}^{i})\right\Vert ^{2}-\Breg_{\pi_{k}^{i}}(\bx_{\star},\bx_{k}^{i}).
\]
\end{lem}
By the co-coercivity property (i.e., Lemma \ref{lem:co-coercivity-cvx}),
the term $\frac{\eta_{k}}{2}\Vert\nabla f_{\pi_{k}^{i}}(\bx_{k}^{i})\Vert^{2}-\Breg_{\pi_{k}^{i}}(\bx_{\star},\bx_{k}^{i})$
can be easily upper bounded by $\eta_{k}\Vert\nabla f_{\pi_{k}^{i}}(\bx_{\star})\Vert^{2}$
(ignoring constant factors) once $\eta_{k}\lesssim1/\Lmax$, which
further yields a desired residual term $\eta_{k}\sigma_{\star}^{2}$
after taking expectations. Therefore, the only difficulty is to relate
$f_{\pi_{k}^{i}}(\bx_{k}^{i})$ to $f(\bx_{k}^{i})$, which is, again,
the main challenge in the analysis of $\ShufflingSGD$.

To address the issue mentioned, the prior work of \citet{pmlr-v97-nagaraj19a}
applies an argument based on Wasserstein distance, which additionally
requires the Lipschitz continuity of each $f_{i}$. In comparison,
we tackle this problem by establishing the following new inequality
in Lemma \ref{lem:I-core}.
\begin{lem}
\label{lem:I-core}Under Assumptions \ref{assu:convexity} and \ref{assu:smoothness},
suppose $\RR$ is employed with $\eta_{k}\leq\frac{1}{6\Lmax},\forall k\in\left[K\right]$,
then for any $k\in\left[K\right]$ and $i\in\left[n\right]$, $\ShufflingSGD$
(Algorithm \ref{alg:Shuffling-SGD}) guarantees that
\[
\E\left[f(\bx_{k}^{i})\right]\leq\E\left[f_{\pi_{k}^{i}}(\bx_{k}^{i})\right]+\eta_{k}\E\left[\left\Vert \nabla f_{\pi_{k}^{i}}(\bx_{k}^{i})\right\Vert ^{2}\right]+\frac{4\eta_{k}}{3n}\sum_{j<i}\E\left[\left\Vert \nabla f_{\pi_{k}^{j}}(\bx_{k}^{j})\right\Vert ^{2}\right].
\]
\end{lem}
\begin{rem}
Lemma \ref{lem:I-core} is stronger than the existing bound of \citet{pmlr-v97-nagaraj19a}
derived via the Wasserstein distance, since imposing the additional
condition $\left\Vert \nabla f_{i}(\bx)\right\Vert \leq G$, as in
\citet{pmlr-v97-nagaraj19a}, recovers their Lemma 4.
\end{rem}
Lemma \ref{lem:I-core} provides a novel inequality that measures
the difference between $\E\left[f(\bx_{k}^{i})\right]$ and $\E[f_{\pi_{k}^{i}}(\bx_{k}^{i})]$
by the second moment of the stochastic gradients up to time $i$.
Note that the second term on the R.H.S. can be absorbed by the R.H.S.
of the inequality in Lemma \ref{lem:I-descent}. For the remaining
term, the coefficient $\eta_{k}/n$ is key to the final proof, which
ensures that the accumulated error in one epoch is controlled by $\E[\eta_{k}\sum_{i=1}^{n}\Vert\nabla f_{\pi_{k}^{i}}(\bx_{k}^{i})\Vert^{2}]$.

The proof of Lemma \ref{lem:I-core} builds on the two core results,
Lemmas \ref{lem:core-function} and \ref{lem:core-stability}, presented
before. To save space, we defer it to Appendix \ref{sec:Bound-I}.

\paragraph{Final proof.}

With Lemmas \ref{lem:I-descent} and \ref{lem:I-core} stated above,
we are ready to prove Theorem \ref{thm:I-rate}.

\begin{proof}[Proof of Theorem \ref{thm:I-rate}]
We sum the inequality in Lemma \ref{lem:I-descent} from $i=1$ to
$n$ and use $\bx_{k+1}^{1}=\bx_{k}^{n+1}$ to obtain
\[
\sum_{i=1}^{n}f_{\pi_{k}^{i}}(\bx_{k}^{i})-f_{\pi_{k}^{i}}(\bx_{\star})=\frac{\left\Vert \bx_{k}^{1}-\bx_{\star}\right\Vert ^{2}-\left\Vert \bx_{k+1}^{1}-\bx_{\star}\right\Vert ^{2}}{2\eta_{k}}+\frac{\eta_{k}}{2}\sum_{i=1}^{n}\left\Vert \nabla f_{\pi_{k}^{i}}(\bx_{k}^{i})\right\Vert ^{2}-\sum_{i=1}^{n}\Breg_{\pi_{k}^{i}}(\bx_{\star},\bx_{k}^{i}).
\]
Take expectations on both sides and note that $\E\left[f_{\pi_{k}^{i}}(\bx_{\star})\right]=f_{\star}$
to yield
\begin{align}
\sum_{i=1}^{n}\E\left[f_{\pi_{k}^{i}}(\bx_{k}^{i})-f_{\star}\right]= & \frac{\E\left[\left\Vert \bx_{k}^{1}-\bx_{\star}\right\Vert ^{2}\right]-\E\left[\left\Vert \bx_{k+1}^{1}-\bx_{\star}\right\Vert ^{2}\right]}{2\eta_{k}}\nonumber \\
 & +\E\left[\sum_{i=1}^{n}\frac{\eta_{k}}{2}\left\Vert \nabla f_{\pi_{k}^{i}}(\bx_{k}^{i})\right\Vert ^{2}-\Breg_{\pi_{k}^{i}}(\bx_{\star},\bx_{k}^{i})\right].\label{eq:I-rate-1}
\end{align}

Next, we invoke Lemma \ref{lem:I-core} and sum it up from $i=1$
to $n$ to have
\begin{equation}
\sum_{i=1}^{n}\E\left[f(\bx_{k}^{i})\right]\leq\sum_{i=1}^{n}\E\left[f_{\pi_{k}^{i}}(\bx_{k}^{i})\right]+\E\left[\frac{7\eta_{k}}{3}\sum_{i=1}^{n}\left\Vert \nabla f_{\pi_{k}^{i}}(\bx_{k}^{i})\right\Vert ^{2}\right].\label{eq:I-rate-2}
\end{equation}
Combine (\ref{eq:I-rate-1}) and (\ref{eq:I-rate-2}) to obtain
\begin{align}
\sum_{i=1}^{n}\E\left[f(\bx_{k}^{i})-f_{\star}\right]\leq & \frac{\E\left[\left\Vert \bx_{k}^{1}-\bx_{\star}\right\Vert ^{2}\right]-\E\left[\left\Vert \bx_{k+1}^{1}-\bx_{\star}\right\Vert ^{2}\right]}{2\eta_{k}}\nonumber \\
 & +\E\left[\sum_{i=1}^{n}\frac{17\eta_{k}}{6}\left\Vert \nabla f_{\pi_{k}^{i}}(\bx_{k}^{i})\right\Vert ^{2}-\Breg_{\pi_{k}^{i}}(\bx_{\star},\bx_{k}^{i})\right].\label{eq:I-rate-3}
\end{align}

One more step, we observe that 
\begin{align}
\left\Vert \nabla f_{\pi_{k}^{i}}(\bx_{k}^{i})\right\Vert ^{2} & \leq\frac{18}{17}\left\Vert \nabla f_{\pi_{k}^{i}}(\bx_{k}^{i})-\nabla f_{\pi_{k}^{i}}(\bx_{\star})\right\Vert ^{2}+18\left\Vert \nabla f_{\pi_{k}^{i}}(\bx_{\star})\right\Vert ^{2}\nonumber \\
 & \overset{(a)}{\leq}\frac{36}{17}L_{\pi_{k}^{i}}\Breg_{\pi_{k}^{i}}(\bx_{\star},\bx_{k}^{i})+18\left\Vert \nabla f_{\pi_{k}^{i}}(\bx_{\star})\right\Vert ^{2}\nonumber \\
\Rightarrow\E\left[\frac{17\eta_{k}}{6}\sum_{i=1}^{n}\left\Vert \nabla f_{\pi_{k}^{i}}(\bx_{k}^{i})\right\Vert ^{2}\right] & \overset{(b)}{\leq}\E\left[\sum_{i=1}^{n}\Breg_{\pi_{k}^{i}}(\bx_{\star},\bx_{k}^{i})\right]+51\eta_{k}n\sigma_{\star}^{2},\label{eq:I-rate-4}
\end{align}
where $(a)$ is by Lemma \ref{lem:co-coercivity-cvx} and $(b)$ is
due to $\eta_{k}\leq\frac{1}{6\Lmax}$ and $\E\left[\left\Vert \nabla f_{\pi_{k}^{i}}(\bx_{\star})\right\Vert ^{2}\right]=\sigma_{\star}^{2},\forall k\in\left[K\right],i\in\left[n\right]$. 

Finally, we plug (\ref{eq:I-rate-4}) back into (\ref{eq:I-rate-3}),
multiply both sides by $\eta_{k}$, sum over $k=1$ to $K$, divide
both sides by $n\sum_{k=1}^{K}\eta_{k}$, apply the convexity of $f$,
and use the definition of $\bar{\bx}_{K}$ in (\ref{eq:avg-x}) to
conclude.
\end{proof}

\subsection{Bound II: Always Better than SGD under Small Stepsize\label{subsec:Bound-II}}

This subsection presents the other rate of $\ShufflingSGD$ under
$\RR$, as shown in Theorem \ref{thm:II-rate} below.
\begin{thm}
\label{thm:II-rate}Under Assumptions \ref{assu:minimizer}, \ref{assu:convexity},
and \ref{assu:smoothness}, suppose $\RR$ is employed with $\eta_{k}\leq\frac{1}{2\Lmax},\forall k\in\left[K\right]$,
then $\ShufflingSGD$ (Algorithm \ref{alg:Shuffling-SGD}) guarantees
that
\[
\E\left[f(\bar{\bx}_{K})-f_{\star}\right]\leq\frac{6\left\Vert \bx_{1}^{1}-\bx_{\star}\right\Vert ^{2}}{n\sum_{k=1}^{K}\eta_{k}}+\frac{8\sum_{k=1}^{K}\eta_{k}^{3}n\Lavg\sigma_{\star}^{2}}{\sum_{k=1}^{K}\eta_{k}},
\]
where $\bar{\bx}_{K}$ is defined in (\ref{eq:avg-x}).
\end{thm}

\paragraph{Discussion on Theorem \ref{thm:II-rate}.}

Readers familiar with the literature on shuffling gradient methods
may readily figure out that the rate given in Theorem \ref{thm:II-rate}
perfectly matches the known bound for $\ShufflingSGD$ under $\RR$
in smooth convex optimization (e.g., \citet{NEURIPS2020_c8cc6e90,JMLR:v22:20-1238}).
However, we emphasize a key difference here, that is, the stepsize
in our Theorem \ref{thm:II-rate} is allowed to satisfy $\eta_{k}\lesssim1/\Lmax$,
in contrast to all existing results that require $\eta_{k}$ to be
at most inversely proportional to $n$.

More importantly, in the setting of nonuniform $L_{i}$ considered
in this work, the largest threshold on the stepsize in the literature
that guarantees a rate similar to Theorem \ref{thm:II-rate} is in
the order of $1/(n\sqrt{\Lavg\Lmax})$ \citep{pmlr-v235-liu24cg,NEURIPS2024_84d39572}.
But as indicated by our Theorem \ref{thm:II-rate}, the superiority
of $\RR$ over $\SGD$ already exists once $\eta_{k}\lesssim1/(n\Lavg)$.
Especially, this improvement can be significant when a dominant smoothness
parameter exists, leading to $\Lmax\approx n\Lavg$.

Therefore, Theorem \ref{thm:II-rate} is the first result to extend
the known bound in smooth convex optimization to any reasonable stepsize
while preserving the favorable property of Algorithm \ref{alg:Shuffling-SGD},
i.e., $\ShufflingSGD$ under $\RR$ provably converges faster than
$\SGD$ when the stepsize is sufficiently small.

\paragraph{Analysis.}

The roadmap for establishing Theorem \ref{thm:II-rate} differs wildly
from that used before to prove Theorem \ref{thm:I-rate}. This time,
our proof strategy is to check how close $\ShufflingSGD$ can be to
Gradient Descent. More concretely, we will view each epoch of Algorithm
\ref{alg:Shuffling-SGD} (containing $n$ iterations) as a single
step and analyze the progress made by it at once. This kind of approach
has appeared in different previous works (e.g., \citet{NEURIPS2020_c8cc6e90,JMLR:v22:20-1238})
and always yields a convergence rate in the order of $1/(n^{1/3}K^{2/3})$,
faster than $\SGD$ when $K$ is large.

However, all works that follow the idea described above share the
same issue, that is, they require the stepsize to be in the order
of $1/n$, which is, however, not enough for our purpose. In the following,
we develop a new analysis to bypass this critical obstacle.

We now formally begin the analysis by introducing another virtual
sequence, defined as follows, for any $k\in\left[K\right]$,
\begin{eqnarray}
\by_{k}^{i+1}\defeq\by_{k}^{i}-\eta_{k}\nabla f_{\pi_{k}^{i}}(\bx_{\star}),\forall i\in\left[n\right], & \by_{k+1}^{1}\defeq\by_{k}^{n+1}, & \text{where}\enspace\by_{1}^{1}\defeq\bx_{\star}.\label{eq:main-virtual-y}
\end{eqnarray}
Under the above definition, and noting that $\nabla f(\bx_{\star})=\bzero$,
one can find
\[
\by_{k}^{n+1}=\by_{k}^{1}-\eta_{k}\sum_{i=1}^{n}\nabla f_{\pi_{k}^{i}}(\bx_{\star})=\by_{k}^{1}-\eta_{k}n\nabla f(\bx_{\star})=\by_{k}^{1},\forall k\in\left[K\right].
\]
Combine the above line and $\by_{1}^{1}=\bx_{\star}$ as defined in
(\ref{eq:main-virtual-y}) to have
\begin{equation}
\by_{k}^{n+1}=\by_{k}^{1}=\bx_{\star},\forall k\in\left[K\right].\label{eq:main-virtual-y-prop}
\end{equation}
The above virtual sequence is inspired by the work of \citet{NEURIPS2020_c8cc6e90},
which, as far as we know, was the first to propose a similar term
under the constant stepsize. Here, we slightly extend their idea to
accommodate the case where the stepsize can depend on the current
epoch number.
\begin{rem}
We note that \citet{NEURIPS2020_c8cc6e90} introduced the virtual
sequence to handle the case of individual strong convexity, i.e.,
each $f_{i}$ is required to be strongly convex. However, in our
setting, only individual convexity is assumed. This means that their
proof cannot be applied. As such, our analysis substantially departs
from the existing approach.
\end{rem}
Equipped with the new virtual sequence introduced above, we first
present the following Lemma \ref{lem:II-descent}.
\begin{lem}
\label{lem:II-descent}Under Assumptions \ref{assu:minimizer}, \ref{assu:convexity},
and \ref{assu:smoothness}, suppose $\eta_{k}\leq\frac{1}{2\Lmax},\forall k\in\left[K\right]$,
then for any $k\in\left[K\right]$, $\ShufflingSGD$ (Algorithm \ref{alg:Shuffling-SGD})
guarantees that
\[
\eta_{k}\sum_{i=1}^{n}\Breg_{\pi_{k}^{i}}(\bx_{k}^{i},\bx_{\star})\leq\left\Vert \bx_{k}^{1}-\bx_{\star}\right\Vert ^{2}-\left\Vert \bx_{k+1}^{1}-\bx_{\star}\right\Vert ^{2}+2\eta_{k}\sum_{i=1}^{n}\Breg_{\pi_{k}^{i}}(\by_{k}^{i},\bx_{\star}),
\]
where $\by_{k}^{i}$ is defined in (\ref{eq:main-virtual-y}). 
\end{lem}
As discussed earlier, we intentionally treat each epoch of $\ShufflingSGD$
as a single step. Hence, compared with Lemma \ref{lem:I-descent}
used to prove the first bound, Lemma \ref{lem:II-descent} is in a
different flavor, which shows the progress made by Algorithm \ref{alg:Shuffling-SGD}
over an entire epoch.

Based on the form of Lemma \ref{lem:II-descent}, two tasks naturally
arise. The first is to lower bound the term $\sum_{i=1}^{n}\Breg_{\pi_{k}^{i}}(\bx_{k}^{i},\bx_{\star})$
on the L.H.S. by $\sum_{i=1}^{n}\Breg(\bx_{k}^{i},\bx_{\star})=\sum_{i=1}^{n}f(\bx_{k}^{i})-f_{\star}$.
The second is to upper bound the residual term $\sum_{i=1}^{n}\Breg_{\pi_{k}^{i}}(\by_{k}^{i},\bx_{\star})$
on the R.H.S.

We address the first task by establishing the following Lemma \ref{lem:II-core},
a novel inequality upper bounding each $\E\left[\Breg(\bx_{k}^{i},\bx_{\star})\right]$
in terms of $\E\left[\Breg_{\pi_{k}^{i}}(\bx_{k}^{i},\bx_{\star})\right]$,
$\frac{1}{n}\sum_{j<i}\E\left[\Breg_{\pi_{k}^{j}}(\bx_{k}^{j},\bx_{\star})\right]$,
and $\sigma_{\star}^{2}$.
\begin{lem}
\label{lem:II-core}Under Assumptions \ref{assu:minimizer}, \ref{assu:convexity},
and \ref{assu:smoothness}, suppose $\RR$ is employed with $\eta_{k}\leq\frac{1}{2\Lmax},\forall k\in\left[K\right]$,
then for any $k\in\left[K\right]$ and $i\in\left[n\right]$, $\ShufflingSGD$
(Algorithm \ref{alg:Shuffling-SGD}) guarantees that
\[
\E\left[\Breg(\bx_{k}^{i},\bx_{\star})\right]\leq2\E\left[\Breg_{\pi_{k}^{i}}(\bx_{k}^{i},\bx_{\star})\right]+\frac{4}{n}\sum_{j<i}\E\left[\Breg_{\pi_{k}^{j}}(\bx_{k}^{j},\bx_{\star})\right]+12\eta_{k}^{2}(i-1)\Lavg\sigma_{\star}^{2}.
\]
\end{lem}
To the best of our knowledge, we are the first to obtain an inequality
in such a form. Intuitively, Lemma \ref{lem:II-core} says that each
$\E\left[\Breg(\bx_{k}^{i},\bx_{\star})\right]$ differs from its
stochastic counterpart $\E\left[\Breg_{\pi_{k}^{i}}(\bx_{k}^{i},\bx_{\star})\right]$
by at most a multiplicative constant, together with an average of
the preceding terms $\E\left[\Breg_{\pi_{k}^{j}}(\bx_{k}^{j},\bx_{\star})\right]$
satisfying $j<i$ (not an exact average due to the coefficient $4/n$),
plus an additional term involving $\sigma_{\star}^{2}$. Summing the
inequality in Lemma \ref{lem:II-core} from $i=1$ to $n$ yields
a meaningful lower bound on $\E\left[\sum_{i=1}^{n}\Breg_{\pi_{k}^{i}}(\bx_{k}^{i},\bx_{\star})\right]$
based on $\E\left[\sum_{i=1}^{n}\Breg(\bx_{k}^{i},\bx_{\star})\right]$
and $\sigma_{\star}^{2}$. 

The proof of Lemma \ref{lem:II-core} is rather technical, so we skip
the discussion here. For details, we kindly refer the interested reader
to Appendix \ref{sec:Bound-II}.

Lastly, we need to bound the residual term $\sum_{i=1}^{n}\Breg_{\pi_{k}^{i}}(\by_{k}^{i},\bx_{\star})$
in Lemma \ref{lem:II-descent}. The inequality obtained is given in
Lemma \ref{lem:II-residual} below.
\begin{lem}
\label{lem:II-residual}Under Assumptions \ref{assu:minimizer} and
\ref{assu:smoothness}, suppose $\RR$ is employed, then for any $k\in\left[K\right]$,
$\ShufflingSGD$ (Algorithm \ref{alg:Shuffling-SGD}) guarantees that
\[
\E\left[\sum_{i=1}^{n}\Breg_{\pi_{k}^{i}}(\by_{k}^{i},\bx_{\star})\right]\leq\frac{\eta_{k}^{2}n^{2}\Lavg\sigma_{\star}^{2}}{6},
\]
where $\by_{k}^{i}$ is defined in (\ref{eq:main-virtual-y}).
\end{lem}
Lemma \ref{lem:II-residual} can be derived in a relatively easy way,
as it can be deduced from existing works. In particular, thanks to
\citet{pmlr-v235-liu24cg}, we obtain a bound that depends only on
the average smoothness parameter $\Lavg$ rather than the maximum
smoothness parameter $\Lmax$.

\paragraph{Final proof.}

Armed with Lemmas \ref{lem:II-descent}, \ref{lem:II-core}, and \ref{lem:II-residual}
above, we are finally able to prove Theorem \ref{thm:II-rate}.

\begin{proof}[Proof of Theorem \ref{thm:II-rate}]
First, we invoke Lemma \ref{lem:II-descent} to have
\[
\eta_{k}\sum_{i=1}^{n}\Breg_{\pi_{k}^{i}}(\bx_{k}^{i},\bx_{\star})\leq\left\Vert \bx_{k}^{1}-\bx_{\star}\right\Vert ^{2}-\left\Vert \bx_{k+1}^{1}-\bx_{\star}\right\Vert ^{2}+2\eta_{k}\sum_{i=1}^{n}\Breg_{\pi_{k}^{i}}(\by_{k}^{i},\bx_{\star}),
\]
where $\by_{k}^{i}$ is defined in (\ref{eq:main-virtual-y}). Take
expectations on both sides and apply Lemma \ref{lem:II-residual}
to yield
\begin{equation}
\eta_{k}\E\left[\sum_{i=1}^{n}\Breg_{\pi_{k}^{i}}(\bx_{k}^{i},\bx_{\star})\right]\leq\E\left[\left\Vert \bx_{k}^{1}-\bx_{\star}\right\Vert ^{2}\right]-\E\left[\left\Vert \bx_{k+1}^{1}-\bx_{\star}\right\Vert ^{2}\right]+\frac{\eta_{k}^{3}n^{2}\Lavg\sigma_{\star}^{2}}{3}.\label{eq:II-rate-1}
\end{equation}

Next, we multiply both sides of the inequality in Lemma \ref{lem:II-core}
by $\eta_{k}$ and sum it up from $i=1$ to $n$ to obtain
\begin{equation}
\eta_{k}\E\left[\sum_{i=1}^{n}\Breg(\bx_{k}^{i},\bx_{\star})\right]\leq6\eta_{k}\E\left[\sum_{i=1}^{n}\Breg_{\pi_{k}^{i}}(\bx_{k}^{i},\bx_{\star})\right]+6\eta_{k}^{3}n^{2}\Lavg\sigma_{\star}^{2}.\label{eq:II-rate-2}
\end{equation}
Combine (\ref{eq:II-rate-1}) and (\ref{eq:II-rate-2}) to have
\begin{align*}
\eta_{k}\E\left[\sum_{i=1}^{n}\Breg(\bx_{k}^{i},\bx_{\star})\right] & \leq6\E\left[\left\Vert \bx_{k}^{1}-\bx_{\star}\right\Vert ^{2}\right]-6\E\left[\left\Vert \bx_{k+1}^{1}-\bx_{\star}\right\Vert ^{2}\right]+8\eta_{k}^{3}n^{2}\Lavg\sigma_{\star}^{2}\\
\Rightarrow\E\left[\sum_{k=1}^{K}\sum_{i=1}^{n}\eta_{k}\Breg(\bx_{k}^{i},\bx_{\star})\right] & \leq6\left\Vert \bx_{1}^{1}-\bx_{\star}\right\Vert ^{2}+\sum_{k=1}^{K}8\eta_{k}^{3}n^{2}\Lavg\sigma_{\star}^{2}.
\end{align*}

Finally, we observe that $\Breg(\bx_{k}^{i},\bx_{\star})=f(\bx_{k}^{i})-f_{\star}$,
divide both sides by $n\sum_{k=1}^{K}\eta_{k}$, apply the convexity
of $f$, and use the definition of $\bar{\bx}_{K}$ in (\ref{eq:avg-x})
to conclude.
\end{proof}

\subsection{Proofs of Theorem \ref{thm:main-RR-cvx} and Corollary \ref{cor:main-RR-cvx-constant-stepsize}\label{subsec:Proof}}

With the previous preparation, Theorem \ref{thm:main-RR-cvx} follows
immediately.

\begin{proof}[Proof of Theorem \ref{thm:main-RR-cvx}]
Combine Theorems \ref{thm:I-rate} and \ref{thm:II-rate} to conclude.
\end{proof}

We next derive Corollary \ref{cor:main-RR-cvx-constant-stepsize}
directly from Theorem \ref{thm:main-RR-cvx}.

\begin{proof}[Proof of Corollary \ref{cor:main-RR-cvx-constant-stepsize}]
With a constant stepsize $\eta_{k}=\eta\leq\frac{1}{6\Lmax},\forall k\in\left[K\right]$,
Theorem \ref{thm:main-RR-cvx} reduces to
\[
\E\left[f(\bar{\bx}_{K})-f_{\star}\right]\leq\frac{6\left\Vert \bx_{1}^{1}-\bx_{\star}\right\Vert ^{2}}{\eta nK}+51\min\left\{ \eta,\eta^{2}n\Lavg\right\} \sigma_{\star}^{2}.
\]
Optimizing the R.H.S. of the above inequality over $0<\eta\leq\frac{1}{6\Lmax}$
yields the desired result.
\end{proof}

%% file: conclusion.tex
\section{Conclusion and Future Work}

In this work, we prove that $\ShufflingSGD$ under $\textsf{RR}$
dominates $\textsf{SGD}$ in smooth convex optimization under any
reasonable stepsize after any finite number of epochs. Our main Theorem
\ref{thm:main-RR-cvx} follows from combining two novel convergence
results, whose analysis may each be of independent interest.

Our work suggests several new directions for future research. From
an upper-bound perspective, our current proof is split into two distinct
parts. It is therefore worthwhile to investigate whether Theorem \ref{thm:main-RR-cvx}
can be obtained via a unified analysis. From a lower-bound perspective,
the only existing hardness result for $\RR$ in smooth convex optimization
by \citet{pmlr-v202-cha23a} is established for the stepsize $\eta_{k}$
at most inversely proportional to $n$ and for the number of epochs
$K$ at least proportional to $n$. As such, providing a complete
characterization of the lower bound for $\RR$ under any reasonable
stepsize and any number of epochs remains an important task for future
work.

%% file: appendix.tex
\section{Additional Related Work\label{sec:Additional-Related-Work}}

This section provides additional discussion of the related work. We
mainly focus on smooth convex optimization under $\RR$. As for nonsmooth
convex optimization under $\RR/\SS/\IG$, the reader could refer to
\citet{kibardin1979decomposition,doi:10.1137/S1052623499362111,NEURIPS2022_7bc4f74e,pmlr-v267-liu25ct}.
See also \citet{NIPS2016_c74d97b0} for $\RR$ under structured problems.

The first breakthrough on $\RR$ is by \citet{gurbuzbalaban2021random}
for smooth strongly convex optimization, which, however, requires
each $f_{i}$ to be quadratic or to have a Lipschitz Hessian. Since
then, extensive studies have emerged. Among them, a series of works
continues to study the convergence behavior of $\RR$ in smooth strongly
convex optimization for quadratic objectives or under other additional
conditions \citep{8514028,pmlr-v97-haochen19a,pmlr-v125-safran20a,pmlr-v119-rajput20a,NEURIPS2020_cb8acb1d,NEURIPS2021_803ef568,rajput2022permutationbased}.

To the best of our knowledge, the first work that drops the strong
convexity assumption is \citet{pmlr-v97-nagaraj19a}, which provides
a convergence rate of $\frac{D^{2}}{\eta nK}+\eta G^{2}$ under the
requirement $\eta\leq\frac{2}{L}$ for $L$-smooth $G$-Lipschitz
convex $f_{i}$, where we remind the reader that $D$ denotes the
distance between the initial point and the optimal solution, $\eta$
represents the stepsize, and $K$ is the number of epochs. However,
this rate cannot reflect any advantage of $\RR$ over standard $\SGD$.
Subsequently, two works \citep{NEURIPS2020_c8cc6e90,JMLR:v22:20-1238}
further remove the extra Lipschitz assumption and establish the bound
$\frac{D^{2}}{\eta nK}+\eta^{2}nL\sigma_{\star}^{2}$ under the condition
$\eta\lesssim\frac{1}{nL}$, where $\sigma_{\star}^{2}$ is the gradient
variance at the optimal solution. This rate is faster than standard
$\SGD$ under its required regime and remains the best bound so far.
In fact, it is unimprovable for small $\eta\lesssim\frac{1}{nL}$
and large $K\gtrsim\frac{nL^{2}D^{2}}{\sigma_{\star}^{2}}$ due to
the lower bound of $(\frac{L\sigma_{\star}^{2}D^{4}}{nK^{2}})^{\frac{1}{3}}$
by \citet{pmlr-v202-cha23a}. Recently, \citet{pmlr-v235-liu24cg}
and \citet{ICLR2025_fea9f93f} extend the above rate from the average
iterate to the last iterate (up to additional polylogarithmic factors).

\section{Summary of Notation}

For readability, we recall and summarize the notation used in the
paper.
\begin{itemize}
\item $f=\frac{1}{n}\sum_{i=1}^{n}f_{i}$ is the objective, where each $f_{i}:\R^{d}\to\R$
is differentiable and convex.
\item $\bx_{\star}$ denotes the minimizer of $f$. $f_{\star}=f(\bx_{\star})$
is the optimal function value. $\sigma_{\star}^{2}=\frac{1}{n}\sum_{i=1}^{n}\left\Vert \nabla f_{i}(\bx_{\star})\right\Vert ^{2}$
is the variance of the gradient at the optimal solution.
\item $L_{i}>0$ is the smoothness parameter of $f_{i}$. $\Lavg=\frac{1}{n}\sum_{i=1}^{n}L_{i}$
is the average smoothness parameter. $\Lmax=\max_{i\in\left[n\right]}L_{i}$
is the maximum smoothness parameter.
\item $\Breg_{h}(\bx,\by)=h(\bx)-h(\by)-\left\langle \nabla h(\by),\bx-\by\right\rangle $
denotes the Bregman divergence induced by a real-valued differentiable
function $h$ (not necessarily convex). We write $\Breg$ (resp. $\Breg_{i}$)
as shorthand for $\Breg_{f}$ (resp. $\Breg_{f_{i}}$).
\end{itemize}

\section{Missing Proofs of Core Lemmas\label{sec:Core-Lemmas}}

This section contains the missing proofs of the two core lemmas presented
in Subsection \ref{subsec:Core-Lemmas}. 

Before providing the proofs, we recall two notions introduced in (\ref{eq:main-virtual-pi})
and (\ref{eq:main-virtual-x}), respectively. Given a permutation
$\pi$ of $\left[n\right]$ and two indices $i,j\in\left[n\right]$
satisfying $j\leq i$, $\pi(i,j)$ is the permutation generated by
exchanging the elements $\pi_{i}$ and $\pi_{j}$ in $\pi$, i.e.,
\begin{equation}
\pi(i,j)=\left(\pi^{1},\mydots,\pi^{j-1},\pi^{i},\pi^{j+1},\mydots,\pi^{i-1},\pi^{j},\pi^{i+1},\mydots,\pi^{n}\right).\label{eq:virtual-pi}
\end{equation}
For any given $k\in\left[K\right]$, $\bx_{k}^{l}(i,j),\forall l\in\left[n+1\right]$
denotes the trajectory of the $k$-th epoch starting from $\bx_{k}^{1}$,
produced by $\ShufflingSGD$, but under the permutation $\pi_{k}(i,j)$,
i.e.,
\begin{eqnarray}
 & \bx_{k}^{l+1}(i,j)=\bx_{k}^{l}(i,j)-\eta_{k}\nabla f_{\pi_{k}^{l}(i,j)}(\bx_{k}^{l}(i,j)),\forall l\in\left[n\right], & \text{where}\enspace\bx_{k}^{1}(i,j)=\bx_{k}^{1}.\label{eq:virtual-x}
\end{eqnarray}

\subsection{Proof of Lemma \ref{lem:core-function}}

\begin{proof}
Given $k\in\left[K\right]$ and $i\in\left[n\right]$, we first have
the decomposition
\[
\ell(\bx_{k}^{i})-\ell_{\pi_{k}^{i}}(\bx_{k}^{i})=\frac{1}{n}\left(\sum_{j<i}\ell_{\pi_{k}^{j}}(\bx_{k}^{i})-\ell_{\pi_{k}^{i}}(\bx_{k}^{i})+\sum_{j\geq i}\ell_{\pi_{k}^{j}}(\bx_{k}^{i})-\ell_{\pi_{k}^{i}}(\bx_{k}^{i})\right).
\]
Note that $\E\left[\ell_{\pi_{k}^{j}}(\bx_{k}^{i})\right]=\E\left[\ell_{\pi_{k}^{i}}(\bx_{k}^{i})\right]$
holds for any $j\in\left\{ i,\mydots,n\right\} $ under $\RR$, since
$\ell_{\pi_{k}^{j}}(\bx_{k}^{i})$ and $\ell_{\pi_{k}^{i}}(\bx_{k}^{i})$
are equal in distribution conditioning on $\pi_{k}^{1}$ to $\pi_{k}^{i-1}$
and $\pi_{1}$ to $\pi_{k-1}$. Therefore, we obtain
\begin{equation}
\E\left[\ell(\bx_{k}^{i})-\ell_{\pi_{k}^{i}}(\bx_{k}^{i})\right]=\frac{1}{n}\sum_{j<i}\E\left[\ell_{\pi_{k}^{j}}(\bx_{k}^{i})-\ell_{\pi_{k}^{i}}(\bx_{k}^{i})\right].\label{eq:core-function-1}
\end{equation}
For any fixed $j\in\left[i\right]$, it is known that $\pi_{k}$ equals
$\pi_{k}(i,j)$ in distribution (e.g., Lemma C.3 of \citet{pmlr-v267-liu25ct}),
which implies that $(\pi_{k},\bx_{k})$ also equals $(\pi_{k}(i,j),\bx_{k}(i,j))$
in distribution, since the trajectory of the $k$-th epoch generated
by $\ShufflingSGD$ is deterministically determined by the permutation
and stepsize. This implies that
\[
\E\left[\ell_{\pi_{k}^{j}}(\bx_{k}^{i})\right]=\E\left[\ell_{\pi_{k}^{j}(i,j)}(\bx_{k}^{i}(i,j))\right]=\E\left[\ell_{\pi_{k}^{i}}(\bx_{k}^{i}(i,j))\right],
\]
where the last step is due to $\pi_{k}^{j}(i,j)=\pi_{k}^{i}$ from
its definition (\ref{eq:main-virtual-pi}). Hence, we finally obtain
\[
\E\left[\ell(\bx_{k}^{i})-\ell_{\pi_{k}^{i}}(\bx_{k}^{i})\right]=\frac{1}{n}\sum_{j<i}\E\left[\ell_{\pi_{k}^{i}}(\bx_{k}^{i}(i,j))-\ell_{\pi_{k}^{i}}(\bx_{k}^{i})\right].
\]
\end{proof}

\subsection{Proof of Lemma \ref{lem:core-stability}}

\begin{proof}
By the definition of $\pi(i,j)$ in (\ref{eq:main-virtual-pi}), we
know
\begin{equation}
\pi_{k}^{l}(i,j)=\pi_{k}^{l},\forall l\notin\left\{ i,j\right\} \Rightarrow\nabla f_{\pi_{k}^{l}(i,j)}=\nabla f_{\pi_{k}^{l}},\forall l\notin\left\{ i,j\right\} .\label{eq:core-stability-1}
\end{equation}
Therefore, given $l\notin\left\{ i,j\right\} $, by the definition
of $\bx_{k}(i,j)$ (see (\ref{eq:main-virtual-x})) and the update
rule of $\ShufflingSGD$, we have
\begin{align}
\left\Vert \bx_{k}^{l+1}(i,j)-\bx_{k}^{l+1}\right\Vert ^{2}= & \left\Vert \bx_{k}^{l}(i,j)-\bx_{k}^{l}-\eta_{k}\left(\nabla f_{\pi_{k}^{l}(i,j)}(\bx_{k}^{l}(i,j))-\nabla f_{\pi_{k}^{l}}(\bx_{k}^{l})\right)\right\Vert ^{2}\nonumber \\
\overset{(\ref{eq:core-stability-1})}{=} & \left\Vert \bx_{k}^{l}(i,j)-\bx_{k}^{l}-\eta_{k}\left(\nabla f_{\pi_{k}^{l}}(\bx_{k}^{l}(i,j))-\nabla f_{\pi_{k}^{l}}(\bx_{k}^{l})\right)\right\Vert ^{2}\nonumber \\
= & \left\Vert \bx_{k}^{l}(i,j)-\bx_{k}^{l}\right\Vert ^{2}+\eta_{k}^{2}\left\Vert \nabla f_{\pi_{k}^{l}}(\bx_{k}^{l}(i,j))-\nabla f_{\pi_{k}^{l}}(\bx_{k}^{l})\right\Vert ^{2}\nonumber \\
 & -2\eta_{k}\left\langle \nabla f_{\pi_{k}^{l}}(\bx_{k}^{l}(i,j))-\nabla f_{\pi_{k}^{l}}(\bx_{k}^{l}),\bx_{k}^{l}(i,j)-\bx_{k}^{l}\right\rangle \nonumber \\
\overset{(a)}{\leq} & \left\Vert \bx_{k}^{l}(i,j)-\bx_{k}^{l}\right\Vert ^{2}+\left(\eta_{k}^{2}-\frac{2\eta_{k}}{L_{\pi_{k}^{l}}}\right)\left\Vert \nabla f_{\pi_{k}^{l}}(\bx_{k}^{l}(i,j))-\nabla f_{\pi_{k}^{l}}(\bx_{k}^{l})\right\Vert ^{2}\nonumber \\
\overset{(b)}{\leq} & \left\Vert \bx_{k}^{l}(i,j)-\bx_{k}^{l}\right\Vert ^{2},\label{eq:core-stability-2}
\end{align}
where $(a)$ is due to Lemma \ref{lem:co-coercivity-cvx} and $(b)$
holds by $\eta_{k}\leq\frac{2}{\Lmax}\Rightarrow\eta_{k}^{2}-\frac{2\eta_{k}}{L_{\pi_{k}^{l}}}\leq0$.

Apply (\ref{eq:core-stability-2}) from $l=j+1$ to $l=i-1$ to obtain
\begin{align*}
 & \left\Vert \bx_{k}^{i}(i,j)-\bx_{k}^{i}\right\Vert ^{2}\leq\left\Vert \bx_{k}^{j+1}(i,j)-\bx_{k}^{j+1}\right\Vert ^{2}\\
= & \left\Vert \bx_{k}^{j}(i,j)-\bx_{k}^{j}-\eta_{k}\left(\nabla f_{\pi_{k}^{j}(i,j)}(\bx_{k}^{j}(i,j))-\nabla f_{\pi_{k}^{j}}(\bx_{k}^{j})\right)\right\Vert ^{2}\\
\overset{(c)}{=} & \eta_{k}^{2}\left\Vert \nabla f_{\pi_{k}^{j}(i,j)}(\bx_{k}^{j})-\nabla f_{\pi_{k}^{j}}(\bx_{k}^{j})\right\Vert ^{2}\overset{(\ref{eq:main-virtual-pi})}{=}\eta_{k}^{2}\left\Vert \nabla f_{\pi_{k}^{i}}(\bx_{k}^{j})-\nabla f_{\pi_{k}^{j}}(\bx_{k}^{j})\right\Vert ^{2},
\end{align*}
where $(c)$ holds by $\bx_{k}^{j}(i,j)=\bx_{k}^{j}$, since $\bx_{k}^{l}(i,j)=\bx_{k}^{l},\forall l\in\left[j\right]$
implied by (\ref{eq:core-stability-2}) and $\bx_{k}^{1}(i,j)\overset{(\ref{eq:main-virtual-x})}{=}\bx_{k}^{1}$
together.
\end{proof}

\section{Missing Proofs of Lemmas for Bound I\label{sec:Bound-I}}

In this section, we provide the missing proofs of the lemmas presented
in Subsection \ref{subsec:Bound-I}, which were used to prove the
first convergence rate in Theorem \ref{thm:I-rate}.

\subsection{Proof of Lemma \ref{lem:I-descent}}

\begin{proof}
Given $k\in\left[K\right]$ and $i\in\left[n\right]$, by the definition
of $\Breg_{\pi_{k}^{i}}$,
\begin{align*}
f_{\pi_{k}^{i}}(\bx_{k}^{i})-f_{\pi_{k}^{i}}(\bx_{\star}) & =\left\langle \nabla f_{\pi_{k}^{i}}(\bx_{k}^{i}),\bx_{k}^{i}-\bx_{\star}\right\rangle -\Breg_{\pi_{k}^{i}}(\bx_{\star},\bx_{k}^{i})\\
 & =\frac{\left\Vert \bx_{k}^{i}-\bx_{\star}\right\Vert ^{2}-\left\Vert \bx_{k}^{i+1}-\bx_{\star}\right\Vert ^{2}}{2\eta_{k}}+\frac{\eta_{k}}{2}\left\Vert \nabla f_{\pi_{k}^{i}}(\bx_{k}^{i})\right\Vert ^{2}-\Breg_{\pi_{k}^{i}}(\bx_{\star},\bx_{k}^{i}),
\end{align*}
where the second step holds by the update rule of Algorithm \ref{alg:Shuffling-SGD}.
\end{proof}

\subsection{Proof of Lemma \ref{lem:I-core}}

\begin{proof}
We apply Lemma \ref{lem:core-function} with $\ell=f$ and $\ell_{i}=f_{i}$
to have
\begin{equation}
\E\left[f(\bx_{k}^{i})-f_{\pi_{k}^{i}}(\bx_{k}^{i})\right]=\frac{1}{n}\sum_{j<i}\E\left[f_{\pi_{k}^{i}}(\bx_{k}^{i}(i,j))-f_{\pi_{k}^{i}}(\bx_{k}^{i})\right],\label{eq:I-core-1}
\end{equation}
where $\bx_{k}(i,j)$ is defined in (\ref{eq:main-virtual-x}). Next,
by the $L_{\pi_{k}^{i}}$-smoothness of $f_{\pi_{k}^{i}}$ (Assumption
\ref{assu:smoothness}), we know
\begin{align}
 & f_{\pi_{k}^{i}}(\bx_{k}^{i}(i,j))-f_{\pi_{k}^{i}}(\bx_{k}^{i})\nonumber \\
\leq & \left\langle \nabla f_{\pi_{k}^{i}}(\bx_{k}^{i}),\bx_{k}^{i}(i,j)-\bx_{k}^{i}\right\rangle +\frac{L_{\pi_{k}^{i}}}{2}\left\Vert \bx_{k}^{i}(i,j)-\bx_{k}^{i}\right\Vert ^{2}\nonumber \\
\overset{(a)}{\leq} & \eta_{k}\left\Vert \nabla f_{\pi_{k}^{i}}(\bx_{k}^{i})\right\Vert \left\Vert \nabla f_{\pi_{k}^{i}}(\bx_{k}^{j})-\nabla f_{\pi_{k}^{j}}(\bx_{k}^{j})\right\Vert +\frac{\eta_{k}^{2}L_{\pi_{k}^{i}}}{2}\left\Vert \nabla f_{\pi_{k}^{i}}(\bx_{k}^{j})-\nabla f_{\pi_{k}^{j}}(\bx_{k}^{j})\right\Vert ^{2}\nonumber \\
\overset{(b)}{\leq} & \eta_{k}\left\Vert \nabla f_{\pi_{k}^{i}}(\bx_{k}^{i})\right\Vert ^{2}+\frac{\eta_{k}+2\eta_{k}^{2}L_{\pi_{k}^{i}}}{4}\left\Vert \nabla f_{\pi_{k}^{i}}(\bx_{k}^{j})-\nabla f_{\pi_{k}^{j}}(\bx_{k}^{j})\right\Vert ^{2}\nonumber \\
\leq & \eta_{k}\left\Vert \nabla f_{\pi_{k}^{i}}(\bx_{k}^{i})\right\Vert ^{2}+\frac{\eta_{k}+2\eta_{k}^{2}L_{\pi_{k}^{i}}}{2}\left(\left\Vert \nabla f_{\pi_{k}^{i}}(\bx_{k}^{j})\right\Vert ^{2}+\left\Vert \nabla f_{\pi_{k}^{j}}(\bx_{k}^{j})\right\Vert ^{2}\right)\nonumber \\
\overset{(c)}{\leq} & \eta_{k}\left\Vert \nabla f_{\pi_{k}^{i}}(\bx_{k}^{i})\right\Vert ^{2}+\frac{2}{3}\eta_{k}\left(\left\Vert \nabla f_{\pi_{k}^{i}}(\bx_{k}^{j})\right\Vert ^{2}+\left\Vert \nabla f_{\pi_{k}^{j}}(\bx_{k}^{j})\right\Vert ^{2}\right),\label{eq:I-core-2}
\end{align}
where $(a)$ is by Cauchy-Schwarz inequality and Lemma \ref{lem:core-stability},
$(b)$ is due to AM-GM inequality, and $(c)$ holds by $\eta_{k}\leq\frac{1}{6\Lmax}$.
Finally, we plug (\ref{eq:I-core-2}) back into (\ref{eq:I-core-1})
to obtain
\begin{align*}
\E\left[f(\bx_{k}^{i})-f_{\pi_{k}^{i}}(\bx_{k}^{i})\right] & \leq\frac{\eta_{k}}{n}\sum_{j<i}\E\left[\left\Vert \nabla f_{\pi_{k}^{i}}(\bx_{k}^{i})\right\Vert ^{2}+\frac{2}{3}\left\Vert \nabla f_{\pi_{k}^{i}}(\bx_{k}^{j})\right\Vert ^{2}+\frac{2}{3}\left\Vert \nabla f_{\pi_{k}^{j}}(\bx_{k}^{j})\right\Vert ^{2}\right]\\
 & \overset{(d)}{=}\frac{\eta_{k}}{n}\sum_{j<i}\E\left[\left\Vert \nabla f_{\pi_{k}^{i}}(\bx_{k}^{i})\right\Vert ^{2}+\frac{4}{3}\left\Vert \nabla f_{\pi_{k}^{j}}(\bx_{k}^{j})\right\Vert ^{2}\right]\\
 & \leq\eta_{k}\E\left[\left\Vert \nabla f_{\pi_{k}^{i}}(\bx_{k}^{i})\right\Vert ^{2}\right]+\frac{4\eta_{k}}{3n}\sum_{j<i}\E\left[\left\Vert \nabla f_{\pi_{k}^{j}}(\bx_{k}^{j})\right\Vert ^{2}\right],
\end{align*}
where $(d)$ is due to $\E\left[\left\Vert \nabla f_{\pi_{k}^{i}}(\bx_{k}^{j})\right\Vert ^{2}\right]=\E\left[\left\Vert \nabla f_{\pi_{k}^{j}}(\bx_{k}^{j})\right\Vert ^{2}\right]$
when $j<i$.
\end{proof}

\section{Missing Proofs of Lemmas for Bound II\label{sec:Bound-II}}

In this section, we provide the missing proofs of the lemmas presented
in Subsection \ref{subsec:Bound-II}, which were used to prove the
second convergence rate in Theorem \ref{thm:II-rate}. 

Before presenting the proofs, we recall the notion introduced in (\ref{eq:main-virtual-y}).
For any $k\in\left[K\right]$, the virtual sequence $\by_{k}^{i},\forall i\in\left[n+1\right]$
follows the equations,
\begin{eqnarray}
\by_{k}^{i+1}=\by_{k}^{i}-\eta_{k}\nabla f_{\pi_{k}^{i}}(\bx_{\star}),\forall i\in\left[n\right], & \by_{k+1}^{1}=\by_{k}^{n+1}, & \text{where}\enspace\by_{1}^{1}=\bx_{\star}.\label{eq:virtual-y}
\end{eqnarray}
As shown in (\ref{eq:main-virtual-y-prop}), the virtual sequence
satisfies that
\begin{equation}
\by_{k}^{n+1}=\by_{k}^{1}=\bx_{\star},\forall k\in\left[K\right].\label{eq:virtual-y-prop}
\end{equation}

\subsection{Proof of Lemma \ref{lem:II-descent}}

\begin{proof}
Given $k\in\left[K\right]$ and $i\in\left[n\right]$, by the update
rule of Algorithm \ref{alg:Shuffling-SGD} and the definition of $\by_{k}^{i+1}$
in (\ref{eq:main-virtual-y}), we have
\begin{align*}
 & \left\Vert \bx_{k}^{i+1}-\by_{k}^{i+1}\right\Vert ^{2}=\left\Vert \bx_{k}^{i}-\by_{k}^{i}-\eta_{k}\left(\nabla f_{\pi_{k}^{i}}(\bx_{k}^{i})-\nabla f_{\pi_{k}^{i}}(\bx_{\star})\right)\right\Vert ^{2}\\
= & \left\Vert \bx_{k}^{i}-\by_{k}^{i}\right\Vert ^{2}+\eta_{k}^{2}\left\Vert \nabla f_{\pi_{k}^{i}}(\bx_{k}^{i})-\nabla f_{\pi_{k}^{i}}(\bx_{\star})\right\Vert ^{2}+2\eta_{k}\left\langle \nabla f_{\pi_{k}^{i}}(\bx_{\star})-\nabla f_{\pi_{k}^{i}}(\bx_{k}^{i}),\bx_{k}^{i}-\by_{k}^{i}\right\rangle \\
= & \left\Vert \bx_{k}^{i}-\by_{k}^{i}\right\Vert ^{2}+\eta_{k}^{2}\left\Vert \nabla f_{\pi_{k}^{i}}(\bx_{k}^{i})-\nabla f_{\pi_{k}^{i}}(\bx_{\star})\right\Vert ^{2}+2\eta_{k}\left(\Breg_{\pi_{k}^{i}}(\by_{k}^{i},\bx_{\star})-\Breg_{\pi_{k}^{i}}(\by_{k}^{i},\bx_{k}^{i})-\Breg_{\pi_{k}^{i}}(\bx_{k}^{i},\bx_{\star})\right)\\
\overset{(a)}{\leq} & \left\Vert \bx_{k}^{i}-\by_{k}^{i}\right\Vert ^{2}+\eta_{k}^{2}\left\Vert \nabla f_{\pi_{k}^{i}}(\bx_{k}^{i})-\nabla f_{\pi_{k}^{i}}(\bx_{\star})\right\Vert ^{2}+2\eta_{k}\left(\Breg_{\pi_{k}^{i}}(\by_{k}^{i},\bx_{\star})-\Breg_{\pi_{k}^{i}}(\bx_{k}^{i},\bx_{\star})\right)\\
\overset{(b)}{\leq} & \left\Vert \bx_{k}^{i}-\by_{k}^{i}\right\Vert ^{2}+2\left(\eta_{k}^{2}L_{\pi_{k}^{i}}-\eta_{k}\right)\Breg_{\pi_{k}^{i}}(\bx_{k}^{i},\bx_{\star})+2\eta_{k}\Breg_{\pi_{k}^{i}}(\by_{k}^{i},\bx_{\star})\\
\overset{(c)}{\leq} & \left\Vert \bx_{k}^{i}-\by_{k}^{i}\right\Vert ^{2}-\eta_{k}\Breg_{\pi_{k}^{i}}(\bx_{k}^{i},\bx_{\star})+2\eta_{k}\Breg_{\pi_{k}^{i}}(\by_{k}^{i},\bx_{\star}),
\end{align*}
where we use $\Breg_{\pi_{k}^{i}}(\by_{k}^{i},\bx_{k}^{i})\geq0$
in $(a)$, apply Lemma \ref{lem:co-coercivity-cvx} to $f_{\pi_{k}^{i}}$
in $(b)$, and notice that $\eta_{k}\leq\frac{1}{2\Lmax}\Rightarrow\eta_{k}^{2}L_{\pi_{k}^{i}}\leq\frac{\eta_{k}}{2}$
in $(c)$. Finally, we sum the above inequality from $i=1$ to $n$,
use $\bx_{k+1}^{1}=\bx_{k}^{n+1}$ and $\by_{k}^{n+1}=\by_{k}^{1}=\bx_{\star}$
(see (\ref{eq:main-virtual-y-prop})), and rearrange terms to complete
the proof.
\end{proof}

\subsection{Proof of Lemma \ref{lem:II-core}}

\begin{proof}
Expanding the definitions of $\Breg$ and $\Breg_{\pi_{k}^{i}}$,
we know
\begin{align*}
\Breg(\bx_{k}^{i},\bx_{\star})-\Breg_{\pi_{k}^{i}}(\bx_{k}^{i},\bx_{\star})= & f(\bx_{k}^{i})-f(\bx_{\star})-\left\langle \nabla f(\bx_{\star}),\bx_{k}^{i}-\bx_{\star}\right\rangle \\
 & -\left(f_{\pi_{k}^{i}}(\bx_{k}^{i})-f_{\pi_{k}^{i}}(\bx_{\star})-\left\langle \nabla f_{\pi_{k}^{i}}(\bx_{\star}),\bx_{k}^{i}-\bx_{\star}\right\rangle \right)\\
= & f(\bx_{k}^{i})-f_{\pi_{k}^{i}}(\bx_{k}^{i})-f(\bx_{\star})+f_{\pi_{k}^{i}}(\bx_{\star})\\
 & -\left\langle \nabla f(\bx_{\star})-\nabla f_{\pi_{k}^{i}}(\bx_{\star}),\bx_{k}^{i}-\bx_{\star}\right\rangle .
\end{align*}
Since $\E\left[f_{\pi_{k}^{i}}(\bx_{\star})\right]=f(\bx_{\star})$
and $\E\left[\nabla f_{\pi_{k}^{i}}(\bx_{\star})\right]=\nabla f(\bx_{\star})$
under $\RR$, after taking expectations on both sides, we obtain
\begin{equation}
\E\left[\Breg(\bx_{k}^{i},\bx_{\star})-\Breg_{\pi_{k}^{i}}(\bx_{k}^{i},\bx_{\star})\right]=\E\left[f(\bx_{k}^{i})-f_{\pi_{k}^{i}}(\bx_{k}^{i})-\left\langle \nabla f(\bx_{\star})-\nabla f_{\pi_{k}^{i}}(\bx_{\star}),\bx_{k}^{i}\right\rangle \right].\label{eq:II-core-1}
\end{equation}
Now, we denote by $\ell_{i}(\bx)\defeq f_{i}(\bx)-\left\langle \nabla f_{i}(\bx_{\star}),\bx\right\rangle ,\forall i\in\left[n\right]$
and $\ell(\bx)\defeq\frac{1}{n}\sum_{i=1}^{n}\ell_{i}(\bx)=f(\bx)-\left\langle \nabla f(\bx_{\star}),\bx\right\rangle $.
Then, (\ref{eq:II-core-1}) implies that,
\begin{equation}
\E\left[\Breg(\bx_{k}^{i},\bx_{\star})-\Breg_{\pi_{k}^{i}}(\bx_{k}^{i},\bx_{\star})\right]=\E\left[\ell(\bx_{k}^{i})-\ell_{\pi_{k}^{i}}(\bx_{k}^{i})\right]\overset{(a)}{=}\frac{1}{n}\sum_{j<i}\E\left[\ell_{\pi_{k}^{i}}(\bx_{k}^{i}(i,j))-\ell_{\pi_{k}^{i}}(\bx_{k}^{i})\right],\label{eq:II-core-2}
\end{equation}
where $(a)$ holds by Lemma \ref{lem:core-function}.

Note that $\ell_{\pi_{k}^{i}}$ is $L_{\pi_{k}^{i}}$-smooth by its
definition and Assumption \ref{assu:smoothness}, we therefore have,
for any $j\in\left[i-1\right]$,
\begin{align}
\ell_{\pi_{k}^{i}}(\bx_{k}^{i}(i,j))-\ell_{\pi_{k}^{i}}(\bx_{k}^{i}) & \leq\left\langle \nabla\ell_{\pi_{k}^{i}}(\bx_{k}^{i}),\bx_{k}^{i}(i,j)-\bx_{k}^{i}\right\rangle +\frac{L_{\pi_{k}^{i}}}{2}\left\Vert \bx_{k}^{i}(i,j)-\bx_{k}^{i}\right\Vert ^{2}\nonumber \\
 & =\left\langle \nabla f_{\pi_{k}^{i}}(\bx_{k}^{i})-\nabla f_{\pi_{k}^{i}}(\bx_{\star}),\bx_{k}^{i}(i,j)-\bx_{k}^{i}\right\rangle +\frac{L_{\pi_{k}^{i}}}{2}\left\Vert \bx_{k}^{i}(i,j)-\bx_{k}^{i}\right\Vert ^{2}\nonumber \\
 & \overset{(b)}{\leq}\frac{\left\Vert \nabla f_{\pi_{k}^{i}}(\bx_{k}^{i})-\nabla f_{\pi_{k}^{i}}(\bx_{\star})\right\Vert ^{2}}{2L_{\pi_{k}^{i}}}+L_{\pi_{k}^{i}}\left\Vert \bx_{k}^{i}(i,j)-\bx_{k}^{i}\right\Vert ^{2}\nonumber \\
 & \overset{(c)}{\leq}\Breg_{\pi_{k}^{i}}(\bx_{k}^{i},\bx_{\star})+L_{\pi_{k}^{i}}\left\Vert \bx_{k}^{i}(i,j)-\bx_{k}^{i}\right\Vert ^{2}\nonumber \\
 & \overset{(d)}{\leq}\Breg_{\pi_{k}^{i}}(\bx_{k}^{i},\bx_{\star})+\eta_{k}^{2}L_{\pi_{k}^{i}}\left\Vert \nabla f_{\pi_{k}^{i}}(\bx_{k}^{j})-\nabla f_{\pi_{k}^{j}}(\bx_{k}^{j})\right\Vert ^{2},\label{eq:II-core-3}
\end{align}
where $(b)$ is by Cauchy-Schwarz inequality and AM-GM inequality,
$(c)$ is due to Lemma \ref{lem:co-coercivity-cvx}, and $(d)$ holds
by Lemma \ref{lem:core-stability}. Furthermore, we can bound
\begin{align*}
 & \left\Vert \nabla f_{\pi_{k}^{i}}(\bx_{k}^{j})-\nabla f_{\pi_{k}^{j}}(\bx_{k}^{j})\right\Vert ^{2}\\
\leq & 4\left\Vert \nabla f_{\pi_{k}^{i}}(\bx_{k}^{j})-\nabla f_{\pi_{k}^{i}}(\bx_{\star})\right\Vert ^{2}+4\left\Vert \nabla f_{\pi_{k}^{j}}(\bx_{k}^{j})-\nabla f_{\pi_{k}^{j}}(\bx_{\star})\right\Vert ^{2}+4\left\Vert \nabla f_{\pi_{k}^{i}}(\bx_{\star})\right\Vert ^{2}+4\left\Vert \nabla f_{\pi_{k}^{j}}(\bx_{\star})\right\Vert ^{2}\\
\leq & 8L_{\pi_{k}^{i}}\Breg_{\pi_{k}^{i}}(\bx_{k}^{j},\bx_{\star})+8L_{\pi_{k}^{j}}\Breg_{\pi_{k}^{j}}(\bx_{k}^{j},\bx_{\star})+4\left\Vert \nabla f_{\pi_{k}^{i}}(\bx_{\star})\right\Vert ^{2}+4\left\Vert \nabla f_{\pi_{k}^{j}}(\bx_{\star})\right\Vert ^{2},
\end{align*}
where the last step is by Lemma \ref{lem:co-coercivity-cvx} again.
Plug the above inequality back into (\ref{eq:II-core-3}) and use
$\eta_{k}\le\frac{1}{2\Lmax}$ to have
\begin{align}
\ell_{\pi_{k}^{i}}(\bx_{k}^{i}(i,j))-\ell_{\pi_{k}^{i}}(\bx_{k}^{i})\leq & \Breg_{\pi_{k}^{i}}(\bx_{k}^{i},\bx_{\star})+2\Breg_{\pi_{k}^{i}}(\bx_{k}^{j},\bx_{\star})+2\Breg_{\pi_{k}^{j}}(\bx_{k}^{j},\bx_{\star})\nonumber \\
 & +4\eta_{k}^{2}L_{\pi_{k}^{i}}\left\Vert \nabla f_{\pi_{k}^{i}}(\bx_{\star})\right\Vert ^{2}+4\eta_{k}^{2}L_{\pi_{k}^{i}}\left\Vert \nabla f_{\pi_{k}^{j}}(\bx_{\star})\right\Vert ^{2}.\label{eq:II-core-4}
\end{align}

Combine (\ref{eq:II-core-2}) and (\ref{eq:II-core-4}) to obtain
\begin{align*}
\E\left[\Breg(\bx_{k}^{i},\bx_{\star})-\Breg_{\pi_{k}^{i}}(\bx_{k}^{i},\bx_{\star})\right]\leq & \frac{1}{n}\sum_{j<i}\E\left[\Breg_{\pi_{k}^{i}}(\bx_{k}^{i},\bx_{\star})+2\Breg_{\pi_{k}^{i}}(\bx_{k}^{j},\bx_{\star})+2\Breg_{\pi_{k}^{j}}(\bx_{k}^{j},\bx_{\star})\right]\\
 & +\frac{4\eta_{k}^{2}}{n}\sum_{j<i}\E\left[L_{\pi_{k}^{i}}\left\Vert \nabla f_{\pi_{k}^{i}}(\bx_{\star})\right\Vert ^{2}+L_{\pi_{k}^{i}}\left\Vert \nabla f_{\pi_{k}^{j}}(\bx_{\star})\right\Vert ^{2}\right].
\end{align*}
When $j<i$, we observe that $\E\left[\Breg_{\pi_{k}^{i}}(\bx_{k}^{j},\bx_{\star})\right]=\E\left[\Breg_{\pi_{k}^{j}}(\bx_{k}^{j},\bx_{\star})\right]$
and the following two inequalities hold
\begin{align*}
\E\left[L_{\pi_{k}^{i}}\left\Vert \nabla f_{\pi_{k}^{i}}(\bx_{\star})\right\Vert ^{2}\right] & =\frac{\sum_{l=1}^{n}L_{l}\left\Vert \nabla f_{l}(\bx_{\star})\right\Vert ^{2}}{n}\leq\Lmax\sigma_{\star}^{2}\leq n\Lavg\sigma_{\star}^{2},\\
\E\left[L_{\pi_{k}^{i}}\left\Vert \nabla f_{\pi_{k}^{j}}(\bx_{\star})\right\Vert ^{2}\right] & =\frac{\sum_{l=1}^{n}\frac{n\Lavg-L_{l}}{n-1}\left\Vert \nabla f_{l}(\bx_{\star})\right\Vert ^{2}}{n}\leq2\Lavg\sigma_{\star}^{2}.
\end{align*}
Put everything together, rearrange terms, and use $1\leq n$ to conclude
the desired inequality.
\end{proof}

\subsection{Proof of Lemma \ref{lem:II-residual}}

\begin{proof}
By smoothness (i.e., Assumption \ref{assu:smoothness}), we have
\[
\Breg_{\pi_{k}^{i}}(\by_{k}^{i},\bx_{\star})\leq\frac{L_{\pi_{k}^{i}}}{2}\left\Vert \by_{k}^{i}-\bx_{\star}\right\Vert ^{2}\overset{(\ref{eq:main-virtual-y})}{=}\frac{L_{\pi_{k}^{i}}}{2}\left\Vert \by_{k}^{i}-\by_{k}^{1}\right\Vert ^{2}\overset{(\ref{eq:main-virtual-y})}{=}\frac{\eta_{k}^{2}L_{\pi_{k}^{i}}}{2}\left\Vert \sum_{j=1}^{i-1}\nabla f_{\pi_{k}^{j}}(\bx_{\star})\right\Vert ^{2}.
\]
Therefore, we can bound 
\[
\E\left[\sum_{i=1}^{n}\Breg_{\pi_{k}^{i}}(\by_{k}^{i},\bx_{\star})\right]\leq\frac{\eta_{k}^{2}}{2}\E\left[\sum_{i=1}^{n}L_{\pi_{k}^{i}}\left\Vert \sum_{j=1}^{i-1}\nabla f_{\pi_{k}^{j}}(\bx_{\star})\right\Vert ^{2}\right]\leq\frac{\eta_{k}^{2}n^{2}\Lavg\sigma_{\star}^{2}}{6},
\]
where the last step is due to Lemma E.1 of \citet{pmlr-v235-liu24cg}
(the constant here is slightly better, since $\nabla f(\bx_{\star})=\bzero$
in our setting leads to a provable improvement).
\end{proof}

%% file: main-COLT.bbl
\begin{thebibliography}{30}
\providecommand{\natexlab}[1]{#1}
\providecommand{\url}[1]{\texttt{#1}}
\expandafter\ifx\csname urlstyle\endcsname\relax
  \providecommand{\doi}[1]{doi: #1}\else
  \providecommand{\doi}{doi: \begingroup \urlstyle{rm}\Url}\fi

\bibitem[Ahn et~al.(2020)Ahn, Yun, and Sra]{NEURIPS2020_cb8acb1d}
Kwangjun Ahn, Chulhee Yun, and Suvrit Sra.
\newblock Sgd with shuffling: optimal rates without component convexity and large epoch requirements.
\newblock In H.~Larochelle, M.~Ranzato, R.~Hadsell, M.F. Balcan, and H.~Lin, editors, \emph{Advances in Neural Information Processing Systems}, volume~33, pages 17526--17535. Curran Associates, Inc., 2020.
\newblock URL \url{https://proceedings.neurips.cc/paper_files/paper/2020/file/cb8acb1dc9821bf74e6ca9068032d623-Paper.pdf}.

\bibitem[Bengio(2012)]{Bengio2012}
Yoshua Bengio.
\newblock \emph{Practical Recommendations for Gradient-Based Training of Deep Architectures}, pages 437--478.
\newblock Springer Berlin Heidelberg, Berlin, Heidelberg, 2012.
\newblock ISBN 978-3-642-35289-8.
\newblock \doi{10.1007/978-3-642-35289-8_26}.
\newblock URL \url{https://doi.org/10.1007/978-3-642-35289-8_26}.

\bibitem[Bottou(2009)]{bottou2009curiously}
L{\'e}on Bottou.
\newblock Curiously fast convergence of some stochastic gradient descent algorithms.
\newblock In \emph{Proceedings of the symposium on learning and data science, Paris}, volume~8, pages 2624--2633. Citeseer, 2009.

\bibitem[Bottou(2012)]{Bottou2012}
L{\'e}on Bottou.
\newblock \emph{Stochastic Gradient Descent Tricks}, pages 421--436.
\newblock Springer Berlin Heidelberg, Berlin, Heidelberg, 2012.
\newblock ISBN 978-3-642-35289-8.
\newblock \doi{10.1007/978-3-642-35289-8_25}.
\newblock URL \url{https://doi.org/10.1007/978-3-642-35289-8_25}.

\bibitem[Bottou et~al.(2018)Bottou, Curtis, and Nocedal]{doi:10.1137/16M1080173}
L\'{e}on Bottou, Frank~E. Curtis, and Jorge Nocedal.
\newblock Optimization methods for large-scale machine learning.
\newblock \emph{SIAM Review}, 60\penalty0 (2):\penalty0 223--311, 2018.
\newblock \doi{10.1137/16M1080173}.
\newblock URL \url{https://doi.org/10.1137/16M1080173}.

\bibitem[Cai and Diakonikolas(2025)]{ICLR2025_fea9f93f}
Xufeng Cai and Jelena Diakonikolas.
\newblock Last iterate convergence of incremental methods as a model of forgetting.
\newblock In Y.~Yue, A.~Garg, N.~Peng, F.~Sha, and R.~Yu, editors, \emph{International Conference on Learning Representations}, volume 2025, pages 102613--102647, 2025.
\newblock URL \url{https://proceedings.iclr.cc/paper_files/paper/2025/file/fea9f93f4cec99f65a8b4d575fc353a8-Paper-Conference.pdf}.

\bibitem[Cai et~al.(2024)Cai, Lin, and Diakonikolas]{NEURIPS2024_84d39572}
Xufeng Cai, Cheuk~Yin Lin, and Jelena Diakonikolas.
\newblock Tighter convergence bounds for shuffled sgd via primal-dual perspective.
\newblock In A.~Globerson, L.~Mackey, D.~Belgrave, A.~Fan, U.~Paquet, J.~Tomczak, and C.~Zhang, editors, \emph{Advances in Neural Information Processing Systems}, volume~37, pages 72475--72524. Curran Associates, Inc., 2024.
\newblock \doi{10.52202/079017-2310}.
\newblock URL \url{https://proceedings.neurips.cc/paper_files/paper/2024/file/84d395725a9b40cb4a49d84478ac24c7-Paper-Conference.pdf}.

\bibitem[Cha et~al.(2023)Cha, Lee, and Yun]{pmlr-v202-cha23a}
Jaeyoung Cha, Jaewook Lee, and Chulhee Yun.
\newblock Tighter lower bounds for shuffling {SGD}: Random permutations and beyond.
\newblock In Andreas Krause, Emma Brunskill, Kyunghyun Cho, Barbara Engelhardt, Sivan Sabato, and Jonathan Scarlett, editors, \emph{Proceedings of the 40th International Conference on Machine Learning}, volume 202 of \emph{Proceedings of Machine Learning Research}, pages 3855--3912. PMLR, 23--29 Jul 2023.
\newblock URL \url{https://proceedings.mlr.press/v202/cha23a.html}.

\bibitem[Garrigos and Gower(2023)]{garrigos2023handbook}
Guillaume Garrigos and Robert~M Gower.
\newblock Handbook of convergence theorems for (stochastic) gradient methods.
\newblock \emph{arXiv preprint arXiv:2301.11235}, 2023.

\bibitem[G{\"u}rb{\"u}zbalaban et~al.(2021)G{\"u}rb{\"u}zbalaban, Ozdaglar, and Parrilo]{gurbuzbalaban2021random}
Mert G{\"u}rb{\"u}zbalaban, Asu Ozdaglar, and Pablo~A Parrilo.
\newblock Why random reshuffling beats stochastic gradient descent.
\newblock \emph{Mathematical Programming}, 186:\penalty0 49--84, 2021.

\bibitem[Haochen and Sra(2019)]{pmlr-v97-haochen19a}
Jeff Haochen and Suvrit Sra.
\newblock Random shuffling beats {SGD} after finite epochs.
\newblock In Kamalika Chaudhuri and Ruslan Salakhutdinov, editors, \emph{Proceedings of the 36th International Conference on Machine Learning}, volume~97 of \emph{Proceedings of Machine Learning Research}, pages 2624--2633. PMLR, 09--15 Jun 2019.
\newblock URL \url{https://proceedings.mlr.press/v97/haochen19a.html}.

\bibitem[Kibardin(1979)]{kibardin1979decomposition}
Vladimir Kibardin.
\newblock Decomposition into functions in the minimization problem.
\newblock \emph{Automation and Remote Control}, 1979, 01 1979.

\bibitem[Koren et~al.(2022)Koren, Livni, Mansour, and Sherman]{NEURIPS2022_7bc4f74e}
Tomer Koren, Roi Livni, Yishay Mansour, and Uri Sherman.
\newblock Benign underfitting of stochastic gradient descent.
\newblock In S.~Koyejo, S.~Mohamed, A.~Agarwal, D.~Belgrave, K.~Cho, and A.~Oh, editors, \emph{Advances in Neural Information Processing Systems}, volume~35, pages 19605--19617. Curran Associates, Inc., 2022.
\newblock URL \url{https://proceedings.neurips.cc/paper_files/paper/2022/file/7bc4f74e35bcfe8cfe43b0a860786d6a-Paper-Conference.pdf}.

\bibitem[Lan(2020)]{lan2020first}
Guanghui Lan.
\newblock \emph{First-order and stochastic optimization methods for machine learning}.
\newblock Springer, 2020.

\bibitem[Liu and Zhou(2024)]{pmlr-v235-liu24cg}
Zijian Liu and Zhengyuan Zhou.
\newblock On the last-iterate convergence of shuffling gradient methods.
\newblock In Ruslan Salakhutdinov, Zico Kolter, Katherine Heller, Adrian Weller, Nuria Oliver, Jonathan Scarlett, and Felix Berkenkamp, editors, \emph{Proceedings of the 41st International Conference on Machine Learning}, volume 235 of \emph{Proceedings of Machine Learning Research}, pages 32471--32508. PMLR, 21--27 Jul 2024.
\newblock URL \url{https://proceedings.mlr.press/v235/liu24cg.html}.

\bibitem[Liu and Zhou(2025)]{pmlr-v267-liu25ct}
Zijian Liu and Zhengyuan Zhou.
\newblock Improved last-iterate convergence of shuffling gradient methods for nonsmooth convex optimization.
\newblock In Aarti Singh, Maryam Fazel, Daniel Hsu, Simon Lacoste-Julien, Felix Berkenkamp, Tegan Maharaj, Kiri Wagstaff, and Jerry Zhu, editors, \emph{Proceedings of the 42nd International Conference on Machine Learning}, volume 267 of \emph{Proceedings of Machine Learning Research}, pages 40152--40193. PMLR, 13--19 Jul 2025.
\newblock URL \url{https://proceedings.mlr.press/v267/liu25ct.html}.

\bibitem[Mishchenko et~al.(2020)Mishchenko, Khaled, and Richtarik]{NEURIPS2020_c8cc6e90}
Konstantin Mishchenko, Ahmed Khaled, and Peter Richtarik.
\newblock Random reshuffling: Simple analysis with vast improvements.
\newblock In H.~Larochelle, M.~Ranzato, R.~Hadsell, M.F. Balcan, and H.~Lin, editors, \emph{Advances in Neural Information Processing Systems}, volume~33, pages 17309--17320. Curran Associates, Inc., 2020.
\newblock URL \url{https://proceedings.neurips.cc/paper_files/paper/2020/file/c8cc6e90ccbff44c9cee23611711cdc4-Paper.pdf}.

\bibitem[Nagaraj et~al.(2019)Nagaraj, Jain, and Netrapalli]{pmlr-v97-nagaraj19a}
Dheeraj Nagaraj, Prateek Jain, and Praneeth Netrapalli.
\newblock {SGD} without replacement: Sharper rates for general smooth convex functions.
\newblock In Kamalika Chaudhuri and Ruslan Salakhutdinov, editors, \emph{Proceedings of the 36th International Conference on Machine Learning}, volume~97 of \emph{Proceedings of Machine Learning Research}, pages 4703--4711. PMLR, 09--15 Jun 2019.
\newblock URL \url{https://proceedings.mlr.press/v97/nagaraj19a.html}.

\bibitem[Nedic and Bertsekas(2001)]{doi:10.1137/S1052623499362111}
Angelia Nedic and Dimitri~P. Bertsekas.
\newblock Incremental subgradient methods for nondifferentiable optimization.
\newblock \emph{SIAM Journal on Optimization}, 12\penalty0 (1):\penalty0 109--138, 2001.
\newblock \doi{10.1137/S1052623499362111}.
\newblock URL \url{https://doi.org/10.1137/S1052623499362111}.

\bibitem[Nesterov et~al.(2018)]{nesterov2018lectures}
Yurii Nesterov et~al.
\newblock \emph{Lectures on convex optimization}, volume 137.
\newblock Springer, 2018.

\bibitem[Nguyen et~al.(2021)Nguyen, Tran-Dinh, Phan, Nguyen, and van Dijk]{JMLR:v22:20-1238}
Lam~M. Nguyen, Quoc Tran-Dinh, Dzung~T. Phan, Phuong~Ha Nguyen, and Marten van Dijk.
\newblock A unified convergence analysis for shuffling-type gradient methods.
\newblock \emph{Journal of Machine Learning Research}, 22\penalty0 (207):\penalty0 1--44, 2021.
\newblock URL \url{http://jmlr.org/papers/v22/20-1238.html}.

\bibitem[Polyak(1987)]{Solr-KOHA-OAI-TEST:19722}
Boris~T. Polyak.
\newblock \emph{Introduction to optimization}.
\newblock New York, Optimization Software, 1987.

\bibitem[Rajput et~al.(2020)Rajput, Gupta, and Papailiopoulos]{pmlr-v119-rajput20a}
Shashank Rajput, Anant Gupta, and Dimitris Papailiopoulos.
\newblock Closing the convergence gap of {SGD} without replacement.
\newblock In Hal~Daumé III and Aarti Singh, editors, \emph{Proceedings of the 37th International Conference on Machine Learning}, volume 119 of \emph{Proceedings of Machine Learning Research}, pages 7964--7973. PMLR, 13--18 Jul 2020.
\newblock URL \url{https://proceedings.mlr.press/v119/rajput20a.html}.

\bibitem[Rajput et~al.(2022)Rajput, Lee, and Papailiopoulos]{rajput2022permutationbased}
Shashank Rajput, Kangwook Lee, and Dimitris Papailiopoulos.
\newblock Permutation-based {SGD}: Is random optimal?
\newblock In \emph{International Conference on Learning Representations}, 2022.
\newblock URL \url{https://openreview.net/forum?id=YiBa9HKTyXE}.

\bibitem[Robbins and Monro(1951)]{10.1214/aoms/1177729586}
Herbert Robbins and Sutton Monro.
\newblock {A Stochastic Approximation Method}.
\newblock \emph{The Annals of Mathematical Statistics}, 22\penalty0 (3):\penalty0 400 -- 407, 1951.
\newblock \doi{10.1214/aoms/1177729586}.
\newblock URL \url{https://doi.org/10.1214/aoms/1177729586}.

\bibitem[Safran and Shamir(2020)]{pmlr-v125-safran20a}
Itay Safran and Ohad Shamir.
\newblock How good is sgd with random shuffling?
\newblock In Jacob Abernethy and Shivani Agarwal, editors, \emph{Proceedings of Thirty Third Conference on Learning Theory}, volume 125 of \emph{Proceedings of Machine Learning Research}, pages 3250--3284. PMLR, 09--12 Jul 2020.
\newblock URL \url{https://proceedings.mlr.press/v125/safran20a.html}.

\bibitem[Safran and Shamir(2021)]{NEURIPS2021_803ef568}
Itay Safran and Ohad Shamir.
\newblock Random shuffling beats sgd only after many epochs on ill-conditioned problems.
\newblock In M.~Ranzato, A.~Beygelzimer, Y.~Dauphin, P.S. Liang, and J.~Wortman Vaughan, editors, \emph{Advances in Neural Information Processing Systems}, volume~34, pages 15151--15161. Curran Associates, Inc., 2021.
\newblock URL \url{https://proceedings.neurips.cc/paper_files/paper/2021/file/803ef56843860e4a48fc4cdb3065e8ce-Paper.pdf}.

\bibitem[Shamir(2016)]{NIPS2016_c74d97b0}
Ohad Shamir.
\newblock Without-replacement sampling for stochastic gradient methods.
\newblock In D.~Lee, M.~Sugiyama, U.~Luxburg, I.~Guyon, and R.~Garnett, editors, \emph{Advances in Neural Information Processing Systems}, volume~29. Curran Associates, Inc., 2016.
\newblock URL \url{https://proceedings.neurips.cc/paper_files/paper/2016/file/c74d97b01eae257e44aa9d5bade97baf-Paper.pdf}.

\bibitem[Sherman et~al.(2021)Sherman, Koren, and Mansour]{NEURIPS2021_107030ca}
Uri Sherman, Tomer Koren, and Yishay Mansour.
\newblock Optimal rates for random order online optimization.
\newblock In M.~Ranzato, A.~Beygelzimer, Y.~Dauphin, P.S. Liang, and J.~Wortman Vaughan, editors, \emph{Advances in Neural Information Processing Systems}, volume~34, pages 2097--2108. Curran Associates, Inc., 2021.
\newblock URL \url{https://proceedings.neurips.cc/paper_files/paper/2021/file/107030ca685076c0ed5e054e2c3ed940-Paper.pdf}.

\bibitem[Ying et~al.(2019)Ying, Yuan, Vlaski, and Sayed]{8514028}
Bicheng Ying, Kun Yuan, Stefan Vlaski, and Ali~H. Sayed.
\newblock Stochastic learning under random reshuffling with constant step-sizes.
\newblock \emph{IEEE Transactions on Signal Processing}, 67\penalty0 (2):\penalty0 474--489, 2019.
\newblock \doi{10.1109/TSP.2018.2878551}.

\end{thebibliography}
